\documentclass[12pt]{article}

\usepackage{latexsym,amssymb,amsgen,amsmath,amsxtra, amsfonts, amsthm, amscd}
\usepackage{bbm}
\usepackage{graphics}

\usepackage{graphicx}

\usepackage{subfigure}

\usepackage{authblk}

\usepackage[usenames]{color}

\usepackage{amsthm}

\definecolor{darkgreen}{rgb}{0.0,0.5,0.0}

\newtheorem{theorem}{Theorem}
\newtheorem{corollary}[theorem]{Corollary}
\theoremstyle{remark}
\newtheorem{definition}[theorem]{Definition}
\theoremstyle{remark}
\newtheorem{example}[theorem]{{\rm Example}}
\newtheorem{algorithm}[theorem]{Algorithm}
\newtheorem{remark}[theorem]{Remark}

\newcommand{\reach}{\mathrm{reach}}

\def\be{\begin{equation}}
\def\ee{\end{equation}}
\def\ba{\begin{array}}
\def\ea{\end{array}}

\def\bs{\begin{slide}}
\def\es{\end{slide}}
\def\bc{\begin{center}}
\def\ec{\end{center}}

\allowdisplaybreaks[1]

\begin{document}

\title{Higher Order Methods for Differential Inclusions }

\author{Sanja Gonzalez \v{Z}ivanovi\'{c} and Pieter Collins}
%\address{Department of Scientific Computing and Control Theory, Centrum Wiskunde en Informatica, Amsterdam, 1098 XG, The Netherlands}
%\eads{\mailto{Sanja.Zivanovic@cwi.nl}, \mailto{Pieter.Collins@cwi.nl}}
\maketitle

%%%%%%%%%%%%%%%%%%%%%%%%%%%%  INTRO %%%%%%%%%%%%%%
\section{Introduction}

 In this paper, we compute reachable sets of differential inclusions,

\be\label{di1}
 \dot x(t)\in F(x(t)),\,\,x(0)=x_0,
\ee

\noindent where $F$ is a continuous set-valued map with compact and convex values.
A solution of the differential inclusion~\eqref{di1} is an
absolutely continuous function $x:[0,T]\rightarrow\mathbb{R}^n$, such that for almost all $t\in [0,T]$,
$x(\cdot)$ is differentiable at $t$ and $\dot x(t)\in F(x(t))$.
The solution set $S_T(x_0) \subset C([0,T],\mathbb{R}^n)$ is defined as
\begin{multline*}
S_T(x_0)=\{x(\cdot) \in C([0,T],\mathbb{R}^n) \,\mid\, x(\cdot)\text{ is a solution of } \dot x(t)\in F(x(t))  \text{ with } x(0)=x_0  \}.
\end{multline*}
The \emph{reachable set} at time $t$, $R(x_0,t) \subset \mathbb{R}^n$, is defined as
\begin{equation*}
R(x_0,t)=\{x(t)\in\mathbb{R}^n\,| x(\cdot) \in S_t(x_0) \} .
\end{equation*}

\noindent In particular, we are interested in higher-order method for computation
of a rigorous over-approximation of the reachable set
of a differential inclusion.

Differential inclusions are generalization of differential equations
having multivalued right-hand sides, see \cite{AC}, \cite{De}, \cite{S}.
They give a mathematical setting for studying differential equations
with discontinuous right-hand sides. In fact, taking a closed, convex hull of the right-hand side,
one obtains a differential inclusion. Solutions of this differential inclusion are known as Fillipov solutions of the original differential equation; see \cite{F}.

One important application area for differential inclusions is control theory.
Suppose we are given an interval $[0,T]$, and absolutely continuous function $x(\cdot)$
which satisfies the inclusion~\eqref{di1}, where $F(x)= f(x,U)=\bigcup_{u\in U} f(x(t),u)$ for almost all $t\in [0,T]$.
It is known that if the set $U$ is compact and separable,
$f$ is continuous, and $f(x,U)$ is convex for all $x$, then there exists a bounded measurable function $u(t)\in U$,
known as admissible control input, such that $x(t)$ is the solution of the control system,
\be\label{controlsystem}
 \dot x(t) = f(x(t),u(t)),\,\,\,x(0)=x_0.
\ee
\noindent The proof of the above is given in \cite{AC}, and with slight changes
in the assumptions, also in \cite{N} or \cite{Li}.
On the other hand, it is easy to see that each solution of
a control system~\eqref{controlsystem} for a given admissible control input
is also a solution of a differential inclusion~\eqref{di1}.
Therefore, if a control system is not completely controllable
one may want to compute reachable sets corresponding to all possible inputs ($u(t)\in U$)
which is equivalent to computing a reachable set of a differential inclusion.

Similarly, we obtain a differential inclusion from a noisy system of differential equations
\be\label{noisysystem}
 \dot x(t) = f(x(t),v(t)),\quad x(0)=x_0,\quad v(t)\in V.
\ee
Although the form of~\eqref{controlsystem} and~\eqref{noisysystem} are identical, the interpretation is different; 
in~\eqref{controlsystem}, the input $u(t)$ can be chosen by the designer, whereas in~\eqref{noisysystem}, the input is determined by the environment.

 Differential inclusions can also arise as reduced models of high-dimensional systems of differential equations.
For example, suppose we have a large-scale system given in the form
of differential equation $\dot x(t)=f(x(t))$. In general, it is very
hard to analyse large-scale systems and most of the times performing model reduction
is necessary. This gives a simplified model in the form of $\dot z(t) = h(z(t)) + e(t)$,
where $|e(t)|< \epsilon$ represents the error that occurred while simplifying the model.

 For reliability purposes many engineering systems require availability of
verification tools. In order to verify a system, we must
guarantee that an approximate solution will contain the actual solution of the system.
If there is uncertainty in the system, lack of controllability, or just a variety of available dynamics,
one needs to use differential inclusion models.
For verification purposes, one needs to compute over-approximations to the set of solutions.
% COMMENT: I don't know where this should go...
%In general, differential inclusions have applications in mechanics, electrical engineering,
%the theory of automatic control, economical, biological, and social macrosystems.

 An important tool in the study of input-affine control systems~\eqref{controlsystem} is based on the Fliess 
expansion~\cite{Fl}, in which the evolution over a time-step $h$ is expanded as a power-series in integrals of the input.
A numerical method based on this approach was given in~\cite{GK}.
The method cannot be directly applied to study noisy systems~\eqref{noisysystem}, since for this problem we 
need to compute the evolution over all possible inputs, and this point is only briefly addressed.

The first result on the computation of the solution set of a differential inclusion
was given in~\cite{PBV}, who considered Lipschitz differential inclusions,
and gave a polyhedral method for obtaining an approximation
of the solution set $S(x_0)$ to an arbitrary known accuracy.
In the case where $F$ is only upper-semicontinuous with compact, convex values,
it is possible to compute arbitrarily accurate over-approximations to the solution set,
as shown in~\cite{CG}.

 Some different techniques and various types of numerical methods have been proposed
as approximations to the solution set of a differential inclusion. For example,
ellipsoidal calculus was used in \cite{KV}, a Lohner-type algorithm
in \cite{KZ}, grid-based methods in \cite{PBV} and \cite{BR}, optimal control
in \cite{BM} and discrete approximations in \cite{DL,Dt,DF}, \cite{G}.
However, these algorithms either do not give rigorous over-approximations, or are
approximations of low-order (e.g. Euler approximations with a first-order single-step truncation error).
Essentially, the only algorithms mentioned above that could give arbitrary accurate error estimates are the ones that use grids.
However, higher order discretization of a state space greatly
affects efficiency of the algorithm. It was noted in \cite{BR} that if one is trying to obtain higher order
error estimates on the solution set of differential inclusions then grid methods should be avoided.

 In order to provide an over-approximation of the reachable set of (\ref{di1}),
we compute solutions of an ``approximate'' system
\[
\dot y(t)= f(y(t),w_k(t)), \quad y(t_k)=x(t_k), \  w_k(\cdot)\in W,
\]

\noindent for $t\in [t_k,t_{k+1}]$, and add the uniform error bound on the difference of the two solutions.
We provide formulas for the local error based on Lipschitz constants and bounds on higher-order derivatives.
The method is based on a Fliess-like expansion, and extends the results of~\cite{GK} by providing error estimates which are valid for all possible inputs.

 We can obtain improved estimates by the use of the logarithmic norm.
The logarithmic norm was introduced independently in \cite{D}, and \cite{L} in order to derive
error estimates to initial value problems, see also \cite{Sg}. Using the logarithmic norm
is advantageous over the use of Lipschitz constant in the sense that the logarithmic norm
can have negative values, and thus, one can distinguish between forward and reverse time integration,
and between stable and unstable systems. The definition of the logarithmic norm and
a theorem on the logarithmic norm estimate is given in Section \ref{Prel}.

 The numerical result given in Section \ref{num} were obtained using the function calculus
implemented in the tool Ariadne~\cite{A} for reachability analysis and verification of hybrid systems.
In particular, we use \emph{polynomial models} for the rigorous approximation of continuous functions.
Polynomial model expresses approximations to a function in the form of a polynomial (defined over a suitably small domain) plus an interval remainder, and are essentially the same as the \emph{Taylor models} of~\cite{RBM}.

 The paper is organized as follows. In Section \ref{Prel}, we give key ingredients
of the theory used. In Section \ref{Di}, we give mathematical setting for obtaining
over-approximations of the reachable sets of a differential inclusion, and propose an algorithm.
In Section \ref{IAS}, we consider differential inclusions in the form of input-affine systems. We derive
the local error, give formulas for obtaining the error of second and third orders, and
show how to obtain the error of higher-orders. We extend the idea of obtaining over-approximations
for input-affine systems to more general
differential inclusions in Section \ref{GDI}. A
numerical example is given in Section \ref{num}.
We conclude the paper with a discussion on the theory proposed in Section \ref{Disc}.

\section{Preliminaries}\label{Prel}

 Below we give several results on differential inclusions and
the computability of their solutions. For further work on the theory of differential inclusions see
\cite{AC}, \cite{De}, \cite{S}, for computability theory see \cite{W}, and for results on computability
of differential inclusions see \cite{PBV}, \cite{CG}.

We canonically use the supremum norm for the vector norm in $\mathbb{R}^n$, i.e.,
for $x \in \mathbb{R}^n$, $\|x\|_\infty=\max \{|x_1|,...,|x_n|\}$.
The corresponding norm for functions $f: D\subset \mathbb{R}^n \rightarrow \mathbb{R}$ is $\|f\|_\infty = \sup_{x\in D}\|f(x)\|_\infty$.
The corresponding matrix norm is
\[
 \|Q\|_{\infty} = \max_{k=1,...,n} \Bigl \{ \sum_{i=1}^n |q_{ki}| \Bigr\}.
\]

Given a square matrix $Q$ and a matrix norm $\|\cdot\|$, the corresponding \emph{logarithmic norm} is
\[
 \lambda(Q)= \lim_{h\rightarrow 0^+} \frac{\|I + hQ\|-1}{h} .
\]
\noindent There are explicit formulas for the logarithmic norm for several matrix norms,
see \cite{HNW}, \cite{D}. The formula for the logarithmic norm corresponding to the uniform matrix norm that we use is
\[
 \lambda_\infty(Q) = \max_k \{ q_{kk} + \sum_{i\neq k} |q_{ki}| \}.
\]

The following theorem on existence of solutions of differential inclusions and its proof can be
found in \cite{De}. Also, a version of the theorem and its proof can be found in \cite{AC}.
\begin{theorem}
Let $D\subset\mathbb{R}^n$ and $F:[0,T]\times D\rightrightarrows \mathbb{R}^n$
be an upper semicontinuous set-valued mapping, with non-empty, compact and convex values.
Assume that $\|F(t,x))\|\le c(1+\|x\|)$, for some constant $c$, is satisfied on $[0,T]$.
Then for every $x_0 \in D$, there exists an
absolutely continuous function $x:[0,T] \rightarrow \mathbb{R}^n$, such that
$x(t_0)=x_0$ and $\dot x(t)\in F(t,x(t))$ for almost all $t\in [0,T]$.
\end{theorem}

\smallskip

\noindent A result on upper-semicomputability of differential inclusions was presented in~\cite{CG}.

\begin{theorem}
Let $F$ be an upper-semicontinuous multivalued function with compact and
convex values. Consider the initial value problem
$\dot x\in F(x)$, $x(0) = x_0$,
where F is defined on some open domain $V\subset \mathbb{R}^n$. Then the solution operator
$x_0 \mapsto S_T(x_0)$ is upper-semicomputable in the following sense:
 – Given an enumerator of all tuples $(L,M_1,...,M_m)$ such that $F(\bar{L})\subset \cup_{i=1}^{m} M_i$, it is possible to enumerate all tuples $(I, J,K_1, . . .,K_k)$
where $I,K_1, . . . , K_m$ are open rational boxes and $J$ is an open rational interval
such that for every $x_0 \in I$, every solution $\xi$ with $\xi(0) = x_0$ satisfies
$\xi(\bar{J})\subset \cup_{i=1}^k K_i$.
\end{theorem}

In other words, it is possible to approximate the reachable sets arbitrarily accurately given a description of the differential inclusion and an arbitrarily accurate description of the initial state.

\smallskip

 The basic construction of our algorithm is based on the following theorem. The theorem and the proof can be found
in \cite[Corollary 1.14.1]{AC}.

\begin{theorem}
 Let $f:X\times U\rightarrow X$ be continuous where $U$ is a compact separable
metric space and assume that there exists an interval $I$ and an absolutely continuous
$x:I\rightarrow \mathbb{R}^n$, such that for almost all $t\in I$,
\[
 \dot x(t)\in f(x(t),U).
\]
\noindent Then there exists a Lebesgue measurable $u:I\rightarrow U$ such that for almost all $t\in I$,
\[
 \dot x(t)= f(x(t),u(t)).
\]

\end{theorem}

\smallskip

 We shall need the multidimensional mean value theorem,
which can be found in standard textbooks on real analysis,
e.g., see \cite{Wa}. We use the following form of the theorem.

\begin{theorem}\label{MVT}
Let $V\subset \mathbb{R}^n$ be open, and suppose that $f:\mathbb{R}^n \rightarrow \mathbb{R}^m$
is differentiable on V. If $x,x+h\in V$ and $L(x;x+h)\subseteq V$, i.e., line between $x$ and $x+h$
belongs to $V$,
\[
 f(x+h)-f(x) = \int_{0}^{1}  Df(z(s)) ds\,\cdot  h\,
\]
\noindent where $Df$ denotes Jacobian matrix of $f$,  $z(s)=x+sh$, and integration is understood component-wise.
\end{theorem}

\smallskip

The following theorem on the logarithmic norm estimate is taken from \cite{HNW}.
\begin{theorem}\label{lnt}
 Let $x(t)$ satisfy differential equation $\dot x(t)=f(t,x(t))$ with $x(t_0)=x_0$, where
$f$ is Lipschitz continuous. Suppose that
there exist functions $l(t)$, $\delta(t)$ and $\rho$ such that
$\lambda(Df(t,z(t)))\le l(t)$ for all $z(t)\in \mathrm{conv}\{x(t),y(t)\}$ and
$\|\dot y(t) - f(t,y(t))\|\le \delta(t)$, $\|x(t_0)-y(t_0)\|\le \rho$.
Then for $t\ge t_0$ we have
\[
\|y(t)-x(t)\| \le e^{\int_{t_0}^t l(s)ds}\left( \rho + \int_{t_0}^{t} e^{-\int_{t_0}^s l(r)dr} \delta(s) ds \right).
\]
\end{theorem}

\smallskip

 In order to numerically compute the reachable set of a differential inclusion, we need a rigorous way of computing with sets and functions in Euclidean space.
A suitable calculus is given by the \emph{Taylor models} defined in~\cite{MB}:
\begin{definition}\label{tm}
Let $f: D \subset \mathbb{R}^v\rightarrow \mathbb{R}$ be a function
that is $(n+1)$ times continuously partially differentiable on an open set
containing the domain $D$. Let $x_0$ be a point in $D$ and $P$ the $n$-th 
order Taylor polynomial of $f$ around $x_0$. Let $I$ be an interval such
that 
Then $(p,I)p$ is called a Taylor model for $f$ if
\[
 f(x) - p(x-x_0) \in I \text{ for all } x\in D
\]
Then we call the pair $(P, I)$ an $n$-th order Taylor model of $f$ around $x_0$ on $D$.
\end{definition}

In Ariadne, we allow arbitrary polynomial approximations, and not just those defined by the Taylor series.
We take $p$ to be a polynomial on the unit domain $[-1,+1]^v$, and pre-compose $p$ by the inverse of
the affine scaling function $s:[-1,+1]^v\rightarrow D$ with $s_i(z_i)=r_i z_i+m_i$.
Instead of using an interval bound for the difference between $f$ and $p$, we take a positive error bound $e$.
We say $(s,p,e)$ is a \emph{scaled polynomial model} for $f$ on the box domain $D$ if $s:[-1,+1]^v\rightarrow D$ is an affine bijection and
\[ \sup_{x\in D} |f(x)-p(s^{-1}(x))| \leq e . \]
In the special case $D=[-1,+1]^v$, the unit box, we speak of a \emph{unit polynomial model} $(p,e)$ satsfying
\( \sup_{z\in [-1,+1]^v} |f(z)-p(z)| \leq e . \)
We use the notation $p\circ s^{-1}\pm e$ to denote the polynomial model $(s,p,e)$.

Polynomial models support a complete function calculus, including the usual arithmetical operations, algebraic and transcendental functions. 
Formally, if $\mathrm{op}$ is an operator on functions, then there is a corresponding operator $\widehat{\mathrm{op}}$ on polynomial models satisfying the property that if $\hat{f}_i$ are polynomial models for $f_i$, $i=1,\ldots,n$ on common domain $D$, then $\widehat{\mathrm{op}}(\hat{f}_1,\ldots,\hat{f}_n)$ is a polynomial model for $\mathrm{op}(f_1,\ldots,f_n)$ on $D$. 
A full description of polynomial models as used in Ariadne is given in~\cite{CNR}.

For the calculuations described in this paper, it is sufficient to consider sets of the form $S=f(D)$ for $D=[-1,+1]^m$ and $f:\mathbb{R}^m\rightarrow\mathbb{R}^n$.
If $p_i\pm e_i$ are unit polynomial models for $f_i$, then
\[\begin{aligned} 
S \subset \widehat{S} &= p([-1,+1]^m) \pm e \\ 
    &\qquad =\{ x \in \mathbb{R}^n \mid x_i = p_i(z) + d_i \text{ for some } z\in[-1,+1]^m \text{ and } d\in\mathbb{R}^n,\ |d_i|\leq e_i \} . 
\end{aligned}\]
Here, $p:[-1,+1]^m\rightarrow \mathbb{R}^n$ is the polynomial with components $p_i$, and $\pm e$ is the set $\prod_{i=1}^{m}[-e_i,+e_i]$.
The set $\widehat{S}$ is an \emph{over-approximation} to $S$.
Note that by defining polynomials $q_i(z,w)=p_i(z)+e_i w_i$, we have
\[ \widehat{S} = p([-1,+1]^m) \pm e \subset q([-1,+1]^{m+n}) \]
yielding an over-approximation as the polynomial image of the unit box without error terms.

\section{Approximation Scheme}\label{Di}

 We consider differential inclusions in the form of noisy differential equations

\be\label{cp}
  \dot x(t) = f(x(t), v(t)),\,\,\, v(t)\in V,
\ee

\noindent where $x:\mathbb{R}\rightarrow \mathbb{R}^n$, $v(\cdot)$ is a bounded measurable function,
$V\subset \mathbb{R}^m$ is a compact convex set, $f$ is continuous and $f(x,V)$ is convex for all $x\in\mathbb{R}^n$.
In order to compute an over-approximation to the reachable set of~\eqref{cp}, we compute
solution set of a different (an approximate) differential equation and
add the uniform error bound on the difference of the two solutions.

\subsection{Single-step approximation}\label{Step}

 Given an initial set of points $X_0$, define
\be\label{rch}
R(X_0,t)=\{x(t) \mid x(\cdot)\, \textrm{is a solution of~\eqref{cp} with}\; x(0)\in X_0\}
\ee

\noindent as the reachable set at time $t$.

 Let $[0,T]$ be an interval of existence of (\ref{cp}). Let $0=t_0,\, t_1,\, \ldots,\, t_{n-1}, t_n=T$
be a partition of $[0,T]$, and let $h_k=t_{k+1}-t_k$.
For $x\in\mathbb{R}^n$ and $v(\cdot)\in L^\infty([t_k,t_{k+1}];\mathbb{R}^m)$,
define $\phi(x_k,v(\cdot))$ to be the point $x(t_{k+1})$ which is the value at time $t_{k+1}$ of the solution of~\eqref{cp} with $x(t_k)=x_k$.

 At each time step we want to compute an over-approximation $R_{k+1}$ to the set
\begin{equation*}
 \reach(R_k,t_{k},t_{k+1})=\{ \phi(x_k,v(\cdot)) \mid x_k\in R_k\text{ and } v(\cdot)\in L^\infty([t_k,t_{k+1}];\mathbb{R}^m) \} .
\end{equation*}

 Since the space of bounded measurable functions is infinite-dimensional, we aim to approximate the set of all solutions by
restricting the disturbances to a finite-dimensional space.
Consider a set of approximating functions $W_k\subset C([t_k,t_{k+1}];\mathbb{R}^m)$ parameterized as $W_k=\{w(a_k,\cdot)\,|\,a_k\in A\subset\mathbb{R}^p\}$, such as $w(a_k,t) = a_{0k} + a_{1k} (t-t_{k+1/2}) / h_k$ where $t_{k+1/2} = t_k+h_k/2 = (t_k+t_{k+1})/2$.
We then need to find an error bound $\epsilon$ such that
\be\label{eps}
 \forall\,v_k\in L^\infty([t_k,t_{k+1}];V),\ \exists\,a_k\in A \text{ s.t. } \| \phi(x_k,v_k(\cdot)) - \phi(x_k,w(a_k,\cdot)) \| \leq \epsilon_k .
\ee
 Note that we do not need to find explicitly infinitely many  $a_k$'s. Instead we need to choose the correct dimension ($\mathbb{R}^p$) and provide bounds on them to get desired error $\epsilon_k$ . 
Setting $\tilde{\phi}(x_k,a_k)=\phi(x_k,w(a_k,\cdot))$, i.e., $\tilde{\phi}$ also denotes the solution
of $\dot x(t)=f(x_k,w(a_k,\cdot))$, with $x(t_k)=x_k$, at $t=t_{k+1}$, we obtain the over-approximation
\begin{equation*}
 R_{k+1} = \{ \tilde{\phi}(x_k,a_k) +  [-\epsilon_k, \epsilon_k]^n \mid x_k\in R_k \text{ and } a_k\in A\}.
\end{equation*}
\noindent Define the approximate system at time step $k$ by
\be\label{cpa}
  \dot y(t) = f(y(t), w_k(a_k,t)), \,\, y_k=y(t_k), \ t\in[t_k,t_{k+1}].
\ee
\noindent We would like to choose ``approximating'' functions $w_k=w(a_k,\cdot):[t_k,t_{k+1}]\rightarrow \mathbb{R}$,
depending on $x(t_k)$ and $v(\cdot)$,
such that the solution of (\ref{cpa}) is an approximation of high order to the solution of (\ref{cp}). 
The desired local error for this paper is at least of $O(h^3)$.
Then we can expect the global error (cumulative error for the time of computation, $[0,T]$) to be roughly of $O(h^2)$.

 Without loss of generality, we assume that $x(t_k)=y(t_k)$ for all
$k\ge 0$. To be precise, initially, we assume $x(t_0)=y(t_0)$.
After obtaining an over-approximation $R_1$, to the solution set at
time $t_1$, we use $R_1$ at the set of initial points of both the original
system~\eqref{cp} and its approximation~\eqref{cpa}
for the next time step. Thus we have $x(t_1)=y(t_1)\in R_1$. We compute $R_2$, and consider it to be
the set of initial points for both equations at time $t_2$.
Proceeding like this, we have $x(t_k)=y(t_k)$, for all $k\ge0$.

 The local error for a time-step consists of two parts. The first part is the analytical error
given by~\eqref{eps}. The second part is the numerical error which is an interval remainder of the 
polynomial model~(see Definition \ref{tm}) representing the solution $\tilde{\phi}(x_k,a_k)$ of $\dot{x}(t)=f(x(t),w_k(a_k,t))$.
We represent the time-$t_k$ reachable set $R_k=\{h_k(s)\,+\,   [-\varepsilon_k,\varepsilon_k]^n \, |\, s\in[-1,+1]^{p_k}\}$,
as a polynomial model whose remainder consists of both numerical and analytical error.
Here, $p_k$ is the number of parameters used in the description of $R_k$.
The inclusion $R(X_0,t_k)\subseteq R_k$ is guaranteed by this approximation scheme.

\medskip
Note that our method only
guarantees a local error of high order at the sequence of rational points $\{t_k\}$ which is {\it a priori} chosen. If one is trying to estimate the error at times $t_k<t<t_{k+1}$ for any $k$ along a \emph{particular} solution, a different formula should be used such as a logarithmic norm estimate based on Theorem~\ref{lnt}.

\subsection{Algorithm for Computing the Reachable Set}\label{Algo}

In this section we present an algorithm for computation of the solution set of~\eqref{di1},
using the single step computation presented earlier.

\begin{algorithm}\mbox{}
Let $R_k=\{ h_k(s)\pm e_k \mid s\in[-1,+1]^{p_k} \}$ be an over-approximation of the set $R(X_0,t_k)$.
To compute an over-approximation $R_{k+1}$ of $R(X_0,t_{k+1})$:
\begin{enumerate}
 \item\label{flwstp} Compute the flow $\tilde{\phi}_k(x_k,a_k)$ of
\[
 \dot x(t) = f(x(t),w_k(a_{k},t)),\,\,x(t_k)=x_k,
 \]
    for $t\in [t_k,t_{k+1}]$, $x_k \in R_k$, and $a_k\in A$.
 \item\label{errstp} Compute the uniform error bound $\varepsilon_k$ for the error of approximating $\dot{x}=f(x(t),v(t))$ by $\dot{x}=f(x(t),w(t))$.
 \item\label{aplystp} Compute the set $R_{k+1}$ which over-approximates $R(x_0,t_{k+1})$ as $R_{k+1} \supset \{\tilde{\phi}(x_k,a_k)\pm \epsilon_k\,|
 \, x_k\in R_k,\, a_k\in A\}$.
 \item\label{redstp} Reduce the number of parameters (if necessary).
 \item\label{spltstp} Split the new obtained domain (if necessary).
\end{enumerate}
\end{algorithm}

Step~\ref{flwstp} of the algorithm produces an approximated flow
in the form $\tilde{\phi}_k(x_k,a_k) \approx \phi(x_k,w(a_k,\cdot))$
which is guaranteed to be valid for all $x_k\in R_k$.
In practice, we cannot represent $\tilde{\phi}$ exactly, and
instead use polynomial model approximation with guaranteed error bound $\hat{\phi}$.
 In Step~\ref{errstp}, we add the uniform error bound $\varepsilon_k$ to make sure an
over-approximation is achieved. In Step~\ref{aplystp}, we compute a new approximating set by applying the
approximated flow to the initial set of points to obtain a solution set
$R_{k+1}=\{ \hat{\phi}(h(s_k)\pm e_k,a_k)\pm\varepsilon_k \}$.
Steps~\ref{redstp} and~\ref{spltstp} are crucial for the efficiency and the accuracy of the algorithm, as explained below.

 It is important to notice that the number of parameters ($a_k$ initially)
grows over the time steps. At each time-step, the number of parameters doubles, unless certain reduction
of parameters is applied. The easiest way to reduce the number of parameters is to replace
the parameter dependency by a uniform error,
but this can have a negative impact on the accuracy.
Another way to reduce number of parameters is using orthogonalization,
though this is only possible for affine approximations using currently known methods.

 It is also of importance to realize that if the approximating set becomes too large,
it may be hard to compute ``good'' approximations to the flow and/or the error.
In this case, we can split the set into smaller pieces, and evolve each piece separately.
This can improve the error, but is of exponential complexity in the state-space dimension.

\section{Input-Affine Systems}\label{IAS}

 In this section, we restrict attention to the input-affine system
\begin{equation}\label{ca}
\dot x(t) = f(x(t)) + \sum_{i=1}^{m} g_i(x(t)) v_i(t); \quad x(t_0)=x_0.
\end{equation}
\noindent  For some $r\ge 1$ which depends on the desired order, we assume that

\begin{itemize}
 \item $f:\mathbb{R}^n\rightarrow \mathbb{R}^n$ is $C^r$ function,
 \item each $g_i:\mathbb{R}^n\rightarrow \mathbb{R}^n$ is $C^r$ function,
 \item $v_i(\cdot)$ is a measurable function such that
  $v_i(t) \in [-V_i,+V_i]$ for some $V_i > 0$.
\end{itemize}

Then the equation~\eqref{cpa} becomes

\begin{equation}\label{caa}
\dot y(t) = f(y(t)) + \sum_{i=1}^{m} g_i(y(t))w_i(a_k,t); \quad y(t_k)=y_k, \  t\in[t_k,t_{k+1}].
\end{equation}

 In what follows, we assume that we have a bound $B$ on the solutions of (\ref{ca}) and (\ref{caa}) for all $t\in [0,T]$.
We take constants $V_i$, $K$, $K_i$, $L$, $L_i$, $H$, $\Lambda$ such that
\be\label{bounds}
\begin{gathered}
|v_i(\cdot)|\le V_i,\,\,\|f(z(t))\|\le K,\,\,\|g_i(z(t))\|\le K_{i}\,\, \lambda(Df(\cdot))\le \Lambda,\\[\jot]
\|Df(z(t))\| \le L,\,\,\|Dg_i(z(t))\|\le L_i,\,\, \|D^2f(z(t))\| \le H,\,\,\|D^2g_i(z(t))\| \le H_i, \\[\jot]
\end{gathered}
\ee
\noindent for each $i=1,...,m$, and for
all $t\in [0,T]$, and $z(\cdot)\in B$. We also set
\[  K'={\displaystyle\sum_{i=1}^{m}} V_i\,K_i, \ \  L'={\sum_{i=1}^{m}} V_i\,L_i \ \  H'={\sum_{i=1}^{m}} V_i\,H_i. \]
 Here, $Df$ denotes the Jacobian matrix, $D^2 f$ denotes the Hessian matrix,
and $\lambda(\cdot)$ denotes the logarithmic norm of a matrix defined in Section \ref{Prel}.

 We proceed to derive higher order estimates on the error by considering several different
cases. In each of the cases, $w_i(a,\cdot)$ is a real valued
finitely-parametrised function with $a\in A\subset \mathbb{R}^N$.
In general, the number of parameters $N$ depends on the number of inputs and the order of error desired.

  In what follows, we write $h_k=t_{k+1}-t_k$, $t_{k+1/2}=t_k + h_k/2 = (t_k+t_{k+1})/2$, and $\hat{q}(t) = \int_{t_k}^{t} q(s)\,ds$.

\subsection{Error derivation}\label{errde}

The single-step error in the difference between $x_{k+1}$ and $y_{k+1}$
is derived as follows. Writing~\eqref{ca} and ~\eqref{caa} as integral equations, we obtain:
\begin{subequations}\label{picard}
\begin{align}
x(t_{k+1}) &= x(t_k) + \int_{t_k}^{t_{k+1}} f(x(t)) + \sum_{i=1}^{m} g_i(x(t)) v_i(t)\,dt ; \\
y(t_{k+1}) &= y(t_k) + \int_{t_k}^{t_{k+1}} f(y(t)) + \sum_{i=1}^{m} g_i(y(t)) w_i(t)\,dt .
\end{align}
\end{subequations}
%-f(y(t))  y(t_k) + \int_{t_k}^{t_{k+1}} f(x(t))-f(y(t)) \,dt \notag\\
%& \hspace{10em} + \sum_{i=1}^{m} \int_{t_k}^{t_{k+1}} g_i(x(t))v_i(t)- g_i(y(t))w_i(t) \,dt \notag\\
%&=  \int_{t_k}^{t_{k+1}} f(x(t))-f(y(t))\,dt \label{firstF}\\
% &\qquad\qquad + \sum_{i=1}^{m} \int_{t_k}^{t_{k+1}}  g_i(x(t))v_i(t)- g_i(y(t))w_i(t) dt.\label{firstG}
%\end{align}
%\end{subequations}
%
%\be\label{firststep}
Since we can take $x(t_k) = y(t_k)$ as explained in Section~\ref{Di}, we obtain
\begin{subequations}\label{first}
\begin{align}
x(t_{k+1})-y(t_{k+1})
%&= x(t_k)-y(t_k) + \int_{t_k}^{t_{k+1}} f(x(t))-f(y(t)) \,dt \notag\\
%& \hspace{10em} + \sum_{i=1}^{m} \int_{t_k}^{t_{k+1}} g_i(x(t))v_i(t)- g_i(y(t))w_i(t) \,dt \notag\\
&=  \int_{t_k}^{t_{k+1}} f(x(t))-f(y(t))\,dt \label{firstF}\\
 &\qquad\qquad + \sum_{i=1}^{m} \int_{t_k}^{t_{k+1}}  g_i(x(t))v_i(t)- g_i(y(t))w_i(t) dt.\label{firstG}
\end{align}
\end{subequations}
%Equations~\eqref{first} can be used directly to derive first-order local error estimates.

\vspace{\baselineskip}

Integrating by parts the term~\eqref{firstF}, we obtain
%\begin{subequations}\label{secondF}
\begin{align*}
 \eqref{firstF}\ & = \Bigl[ (t-t_{k+1/2}) \bigl(f(x(t)) - f(y(t))\bigr) \Bigr]_{t_k}^{t_{k+1}} \\
&\hspace{3em} - \int_{t_k}^{t_{k+1}} (t-t_{k+1/2})\frac{d}{dt}\bigl( f(x(t))- f(y(t))\bigr)dt \\[\jot]
 & = (h_k/2) \bigl(f(x(t_{k+1}))-f(y(t_{k+1})) \bigr) \\
&\hspace{3em} - \int_{t_k}^{t_{k+1}} (t-t_{k+1/2}) \bigl( Df(x(t))\dot x(t) - Df(y(t))\dot y(t)\bigr)dt.
\end{align*}
%\end{subequations}

\noindent There two ways that we deal with term~\eqref{firstG}. First we rewrite the term inside the integral as
\begin{align*}
 g_i(x(t))v_i(t)- g_i(y(t))w_i(t) =  (g_i(x(t))-g_i(y(t)))\,w_i(t) + g_i(x(t))\,(v_i(t)-w_i(t)),
\end{align*}
\noindent and then integrate by parts the second term to obtain
\begin{subequations}\label{secondG1}
\begin{align}
\eqref{firstG} &= \sum_{i=1}^{m} \int_{t_k}^{t_{k+1}} (g_i(x(t))-g_i(y(t)))\,w_i(t)\,dt \notag \\
&+ \sum_{i=1}^{m} \Bigl[g_i(x(t))(\hat{v}_i(t)- \hat{w}_i(t))\Bigr]_{t_k}^{t_{k+1}}
-\sum_{i=1}^{m}\int_{t_k}^{t_{k+1}} \frac{d}{dt}\Bigl( g_i(x(t))\Bigr)\,(\hat{v}_i(t)-\hat{w}_i(t))\, dt\notag \\
&=\sum_{i=1}^{m} \int_{t_k}^{t_{k+1}} (g_i(x(t))-g_i(y(t)))\,w_i(t)\,dt \label{secondG1a} \\
&\qquad\qquad+ \sum_{i=1}^{m} g_i(x(t_{k+1}))(\hat{v}_i(t_{k+1})- \hat{w}_i(t_{k+1}))\label{secondG1b}\\
&\qquad\qquad-\sum_{i=1}^{m}\int_{t_k}^{t_{k+1}} Dg_i(x(t))\,\dot x(t)\,(\hat{v}_i(t)-\hat{w}_i(t))\, dt\label{secondG1c}
\end{align}
\end{subequations}
\noindent The second derivation is obtained just by integrating by parts,

\begin{subequations}\label{secondG}
\begin{align}
\eqref{firstG}\ &= \sum_{i=1}^{m} \Bigl[ g_i(x(t))\hat{v}_i(t) - g_i(y(t))\hat{w}_i(t) \Bigr]_{t_k}^{t_{k+1}} \notag \\
&\qquad\qquad - \sum_{i=1}^{m} \int_{t_k}^{t_{k+1}}\frac{d}{dt}\Bigl(g_i(x(t))\Bigr) \hat{v}_i(t) - \frac{d}{dt}\Bigl(g_i(y(t))\Bigr) \hat{w}_i(t)\,\, dt\notag\\
  &= \sum_{i=1}^{m} g_i(x(t_{k+1}))\hat{v}_i(t_{k+1}) - g_i(y(t_{k+1}))\hat{w}_i(t_{k+1})\label{secondGa} \\
&\qquad\qquad - \sum_{i=1}^{m} \int_{t_k}^{t_{k+1}} Dg_i(x(t)) \hat{v}_i(t) \dot{x}(t)- Dg_i(y(t))\hat{w}_i(t)\dot{y}(t) \,\, dt\label{secondGb}
\end{align}
\end{subequations}
\noindent Equations~\eqref{firstF} and~\eqref{secondG1} can be used to derive second-order local error estimates.
\vspace{\baselineskip}

By applying the mean value theorem (Theorem~\ref{MVT}) we obtain
\begin{equation*}
f(x(t_{k+1}))-f(y(t_{k+1}) = \int_{0}^{1} Df(z(s))ds\; \bigl(x(t_{k+1})-y(t_{k+1})\bigr)
\end{equation*}
\noindent Hence
\begin{subequations}\label{secondF}
\begin{align}
\eqref{firstF}\ &=  (h_k/2) \int_{0}^{1} Df(z(s))ds\; \bigl(x(t_{k+1})-y(t_{k+1})\bigr) \label{secondFa}\\
&\hspace{5em} - \int_{t_k}^{t_{k+1}} (t-t_{k+1/2}) \bigl( Df(x(t))\dot x(t) - Df(y(t))\dot y(t)\bigr)dt. \label{secondFb}
\end{align}
\end{subequations}

Separate the second part of the integrand in~\eqref{secondFb} as
\begin{subequations}
\begin{align}
  Df(x(t))\, \dot x(t) - Df(y(t))\,\dot y(t) &= Df(x(t))\,\bigl(\dot x(t) -\dot y(t)\bigr) \label{sepFa}\\[\jot]
    &\qquad\qquad +\bigl(Df(x(t)) - Df(y(t))\bigr)\,\dot y(t) \label{sepFb}
\end{align}
\end{subequations}
The first term of the right-hand-side can be expanded using
\begin{align*}
\dot x(t) -\dot y(t)
%&=f(x(t)) - f(y(t)) + \sum_{i=1}^{m} g_i(x(t))v_i(t)-g_i(y(t))w_i(t)\\
&= f(x(t)) - f(y(t)) +  \sum_{i=1}^{m} \bigl(g_i(x(t))-g_i(y(t))\bigr)w_i(t) \\
&\hspace{12em} + \sum_{i=1}^{m} g_i(x(t))\bigl((v_i(t)-w_i(t)\bigr).
\end{align*}

\noindent Hence we obtain
\begin{subequations}\label{thirdF}
\begin{align}
\eqref{firstF} & = (h_k/2) \int_{0}^{1} Df(z(s))ds\, (x(t_{k+1})-y(t_{k+1}))\label{thirdFa}\\
&\qquad - \int_{t_k}^{t_{k+1}} (t-t_{k+1/2}) \,\, Df(x(t))\,\,(f(x(t)) - f(y(t))) \, dt\label{thirdFb}\\
&\qquad - \sum_{i=1}^{m} \int_{t_k}^{t_{k+1}} (t-t_{k+1/2}) \,\, Df(x(t))\,\,(g_i(x(t))-g_i(y(t)))w_i(t) \,dt\label{thirdFc}\\
&\qquad - \sum_{i=1}^{m} \int_{t_k}^{t_{k+1}} (t-t_{k+1/2}) \,\, Df(x(t))\,\,g_i(x(t))\,\,(v_i(t)-w_i(t)) \, dt,\label{thirdFd} \\
&\qquad - \int_{t_k}^{t_{k+1}} (t-t_{k+1/2}) \,\, (Df(x(t)) - Df(y(t)))\,\,\dot y(t) dt\label{thirdFe}
\end{align}
\end{subequations}
where~\eqref{thirdFa} is~\eqref{secondFa},~(\ref{thirdF}b-d) come from~\eqref{sepFa}, and~\eqref{thirdFe} comes from~\eqref{sepFb}.
Note that for any $C^1$-function $h(x)$ we can write
\begin{equation*}
 |h(x(t)) - h(y(t)) | \leq
      \| Dh(z(t)) \| \cdot | x(t) - y(t) |
\end{equation*}
where $z(t)\in \overline{\mathrm{conv}}\{x(t),y(t)\}$. This will allow us to obtain third-order bounds for terms~(\ref{thirdF}b,c,e).
In order to obtain a third-order estimate for term~\eqref{thirdFd}, a further integration by parts is needed. We obtain:
\begin{subequations}\label{thirdFp}
\begin{align}\addtocounter{equation}{3}
 \eqref{thirdFd} &= - \sum_{i=1}^{m} \Bigl[ Df(x(t)) \, g_i(x(t)) \, {\mbox{\small$\displaystyle\int_{t_k}^{t}$}} (s-t_{k+1/2}) (v(s)-w(s)) ds \Bigr]_{t_k}^{t_{k+1}} \notag \\
&\qquad\begin{aligned} &\qquad + \int_{t_k}^{t_{k+1}} \bigl(D^2f(x(t))\,g_i(x(t))+ Df(x(t))Dg_i(x(t))\bigr)\,\dot x(t) \label{thirdFdp}\\
 &\hspace{14em}   \int_{t_k}^{t} (s-t_{k+1/2})(v_i(s)-w_i(s))ds\ dt .
\end{aligned}
\end{align}
\end{subequations}

\noindent Using similar type of derivation as for the derivation of~\eqref{thirdF}, again using the mean value theorem and integration by parts, we obtain
\begin{subequations}\label{thirdG}
\begin{align}
 \eqref{secondGa}+\eqref{secondGb} &= \sum_{i=1}^{m} \int_{0}^{1} Dg_i(z(s))ds\;\bigl(x(t_{k+1})-y(t_{k+1})\bigr)\hat{w}_i(t_{k+1})\label{thirdGa}\\
&\qquad + \sum_{i=1}^{m} g_i(x_{k+1})\bigl(\hat{v}_i(t_{k+1}) - \hat{w}_i(t_{k+1})\bigr)\label{thirdGb}\\
&\qquad - \sum_{i=1}^{m} \int_{t_k}^{t_{k+1}}  \bigl(Dg_i(x(t)) - Dg_i(y(t))\bigr)\,\dot y(t)\, \hat{w}_i(t) dt\label{thirdGc}\\
&\qquad - \sum_{i=1}^{m} \int_{t_k}^{t_{k+1}}  Dg_i(x(t))\,\bigl(f(x(t)) - f(y(t))\bigr)\, \hat{w}_i(t)\, dt\label{thirdGd}\\
&\qquad - \sum_{i=1}^{m} \int_{t_k}^{t_{k+1}}  Dg_i(x(t))\,f(x(t))\,\bigl(\hat{v}_i(t)-\hat{w}_i(t)\bigr)\label{thirdGe}\\
&\qquad - \sum_{i=1}^{m}\sum_{j=1}^{m} \int_{t_k}^{t_{k+1}}  Dg_i(x(t))\,\bigl(g_j(x(t))-g_j(y(t))\bigr)\,w_j(t)\,  \hat{w}_i(t)\,dt\label{thirdGf}\\
&\qquad - \sum_{i=1}^{m}\sum_{j=1}^{m} \int_{t_k}^{t_{k+1}}  Dg_i(x(t))\,g_j(x(t))\,\bigl(v_j(t)\hat{v}_i(t)-w_j(t)\hat{w}_i(t)\bigr) \, dt. \label{thirdGg}
\end{align}
\end{subequations}
The term~\eqref{thirdGe} can be further integrated by parts to obtain
\begin{subequations}\label{thirdGp}
\begin{align}\addtocounter{equation}{4}
\eqref{thirdGe}&  = - \sum_{i=1}^{m} \Bigl[ Dg_i(x(t)) \, f(x(t)) \, {\mbox{\small$\displaystyle\int_{t_k}^{t}$}} (\hat{v}(s)-\hat{w}(s))ds \Bigr]_{t_k}^{t_{k+1}} \notag\\
 & \qquad + \sum_{i=1}^{m} \int_{t_k}^{t_{k+1}}  \bigl( D^2g_i(x(t))\,f(x(t)) + Dg_i(x(t))\,Df(x(t)) \bigr) \dot{x}(t) \,(\hat{\hat{v}}_i(t)-\hat{\hat{w}}_i(t))\, dt \label{thirdGep}
\end{align}
and the term~\eqref{thirdGg} to obtain
\begin{align}\addtocounter{equation}{1}
\eqref{thirdGg}&= - \sum_{i=1}^{m}\sum_{j=1}^{m} \Bigl[  Dg_i(x(t))\,g_j(x(t))\,{\mbox{\small$\displaystyle\int_{t_k}^{t}$}} \bigl(v_j(s)\hat{v}_i(s)-w_j(s)\hat{w}_i(s)\bigr) ds \Bigr] \notag\\
&\qquad\begin{aligned} &\qquad + \sum_{i=1}^{m}\sum_{j=1}^{m} \int_{t_k}^{t_{k+1}}  \bigl( D^2g_i(x(t))\,g_j(x(t))+Dg_i(x(t))\,Dg_j(x(t)) \bigr)\,\dot{x}(t) \\[-\jot]
 &\hspace{16em} \,{\mbox{\small$\displaystyle\int_{t_k}^{t}$}} \bigl(v_j(s)\hat{v}_i(s)-w_j(s)\hat{w}_i(s)\bigr)ds \ dt.
\end{aligned}
\label{thirdGgp}
\end{align}
\end{subequations}
\noindent Equations~(\ref{thirdF}-\ref{thirdGp}) can be used to derive third-order local error estimates.

\subsection{Local error estimates}\label{formulas}

  We proceed to give formulas for the local error having different assumptions on functions $f(\cdot)$, $g_i(\cdot)$
and $w_i(\cdot)$.
We present necessary and sufficient conditions
for obtaining local errors of $O(h)$, $O(h^2)$, $O(h^3)$, and give a methodology
to obtaining even higher-order errors. In addition, we give formulas for the error calculation in several cases.

\subsubsection{Local error of $O(h)$}

\begin{theorem}\label{case1}
For any $k\ge 0$, and all $i=1,...,m$, if
\begin{itemize}
 \item $f(\cdot)$ is a Lipschitz continuous vector function,
 \item $g_i(\cdot)$ are continuous vector functions, and
 \item $w_i(t)=0$ on $[t_k,t_{k+1}]$,
\end{itemize}
\noindent then the local error is of $O(h)$. Moreover, a formula
for the error bound is:
\begin{equation}\label{orderh}
 \bigl| x(t_{k+1}) - y(t_{k+1})\bigr| \le  h_k\,K'\, \frac{e^{\Lambda h_k} -1}{\Lambda\,h_k}.
\end{equation}
Alternatively, we can use
\begin{equation}\label{orderhalt}
 \bigl| x(t_{k+1}) - y(t_{k+1})\bigr| \le  h_k\,\biggl(2K + K'\biggr) .
\end{equation}
\end{theorem}

\begin{proof} Since $w_i(t)=0$, we have $\dot{y}(t)=f(y(t))$. Using the bounds given in~\eqref{bounds}, we can take $l(t)=\Lambda$ in Theorem~\ref{lnt}. Further,
 and
\begin{align*}
 \biggl\|\dot y(t) - \Bigl(f(y(t)) + \sum_{i=1}^m g_i(y(t))v_i(t)\Bigr)\biggr\| = \biggl\|\sum_{i=1}^m g_i(y(t))v_i(t))\biggr\| \le \sum_{i=1}^{m} K_i\,V_i=K'
\end{align*}
\noindent so we can take $\delta(t)=\sum_{i=1}^{m} K_i\,V_i$. Hence the formula~\eqref{orderh} is obtained directly from Theorem~\ref{lnt}.
Note that $(e^{\Lambda h_k} - 1)/(\Lambda\,h_k) = 1 + \Lambda h_k/2 + \cdots$ is $O(1)$,
so the local error is of $O(h)$. Equation~\eqref{orderhalt} can be obtained by noting that $\sup_{t\in [t_k,t_{k+1}]} ||f(x(t))-f(y(t))||\le 2K$.
\end{proof}

\subsubsection{Local error of $O(h^2)$}

\begin{theorem}\label{case2}
For any $k\ge 0$, and all $i=1,...,m$, if
\begin{itemize}
 \item $f(\cdot)$, $g_i(\cdot)$ are $C^1$ vector functions, and
 \item $w_i(\cdot)$ are bounded measurable functions defined on $[t_k,t_{k+1}]$ which satisfy
  \be\label{se1}
   \int_{t_k}^{t_{k+1}} v_i(t) - w_i(t) \, dt = 0,
  \ee
\end{itemize}
\noindent then an error of $O(h^2)$ is obtained.
\end{theorem}

\begin{proof} To show that the error is of $O(h^2)$, we use equations~(\ref{first},\ref{secondG1}).
The equation~\eqref{firstF} is in the
desired form, i.e., of $O(h^2)$, since we can write
\[
\biggl|\int_{t_k}^{t_{k+1}} f(x(t)) - f(y(t))\,dt\biggr|
\le h \, L\; {\textstyle\sup_{t\in [t_k,t_{k+1}]}}\|x(t)-y(t)\|,
\]
\noindent and $\sup_{t\in (t_k,t_{k+1})}\|x(t)-y(t)\|$ is of $O(h)$
by Theorem \ref{lnt}. Similarly, equations~\eqref{secondG1a} and~\eqref{secondG1c} are of $O(h^2)$.
Note that the equation~\eqref{secondG1b} is zero due to
(\ref{se1}). The theorem is proved.

%\begin{align*}
%\Bigl| \sum_{i=1}^{m} \int_{t_k}^{t_{k+1}} (g_i(x(t))-g_i(y(t)))\,w_i(t)\,dt \Bigr| \le h_k \sum_{i=1}^{m} L_i
%\end{align*}
%\begin{align*}
%&-\sum_{i=1}^{m}\int_{t_k}^{t_{k+1}} Dg_i(x(t))\,\dot x(t)\,(\hat{v}_i(t)-\hat{w}_i(t))\, dt
%\end{align*}
%\noindent which is of $O(h^2)$.
\end{proof}

 In order to be able to compute the errors, we need the bounds on the functions $w_i(\cdot)$. In particular,
we can restrict $w_i(\cdot)$ to belong to certain class of functions, such as polynomial or step functions.

\medskip

\begin{theorem}\label{case2a}
For any $k\ge 0$, and all $i=1,...,m$, if
\begin{itemize}
 \item $f(\cdot)$, $g_i(\cdot)$ are $C^1$ vector functions, and
 \item $w_i(t)$ are real valued, constant functions defined on $[t_k,t_{k+1}]$ by
 $ w_i=\frac{1}{h_k} \int_{t_k}^{t_{k+1}} v_i(t)dt , $
\end{itemize}
\noindent then a formula for calculation of the local error is given by
\begin{align}\label{constantapproximationsecondordererror}
 \|x(t_{k+1}) - y(t_{k+1})\| \leq h_k^2\,\left(\left(K+K'\right)L'/3+2\,K'\,\left(L + L'\right)\,\frac{e^{\Lambda\, h_k}-1}{\Lambda\,h_k}\right).
\end{align}
\end{theorem}

\begin{proof}
To derive~\eqref{constantapproximationsecondordererror}, we obtain $\|x(t_{k+1}) - y(t_{k+1})\|$ from equations~\eqref{firstF} and~\eqref{secondG1}.
Using the bounds given in~\eqref{bounds},
it is immediate that $||\dot{x}||\leq K + \sum_{i+1}^{m} V_i\,K_i$,
and straightforward to show that $|w_i(t)|\leq V_i$ and  $|\hat{v}_i(t)-\hat{w}_i(t)| \leq 2V_i\,h_k$ for $t\in[t_k,t_{k+1}]$.
However, we can get a slighly better bound $|\hat{v}_i(t)-\hat{w}_i(t)|\le V_i\, h_k/2$ by considering the following:
Without loss of generality, assume $t\in[0,h]$, and let
\begin{align*}
 a_i(t)=\frac{1}{t}\,\int_{0}^{t} v_i(s)\,ds,\ \ \
 b_i(t)=\frac{1}{h-t}\,\int_{t}^{h-t} v_i(s)\,ds
\end{align*}
\noindent and define \[w_i(t)=(t\,a_i(t)\, +\,(h-t)\,b_i(t))/h.\]
Then, $w_i=w_i(t)$ is constant for all $t\in [0,h]$.
Notice that $\hat{v}_i(t)=ta(t)$ and $\hat{w}_i(t)=(t/h)(ta(t)+(h-t)b(t))$.
Hence, we have
\begin{align*}
 \hat{v}_i(t)-\hat{w}_i(t) &= t(h-t)(a(t)-b(t))/h,\\
 |\hat{v}_i(t)-\hat{w}_i(t)| &= t(h-t)|a(t)-b(t)|/h \le V_i\,h/2.
\end{align*}
Additionally, we can prove that $\int_{t_k}^{t_{k+1}}|\hat{v}_i(t)-\hat{w}_i(t)|\,dt \leq V_i\,h_k^2 /3$.
Take $z(t)$ to satisfy the differential equation $\dot{z}(t)=f(z(t))$.
From Theorem \ref{lnt}, we have
\[
 \|x(t)-z(t)\|,\|y(t)-z(t)\| \le h_k\,\Bigl(\sum_{i=1}^{m} K_i\,V_i\Bigr)\,\frac{e^{\Lambda h_k} -1}{\Lambda\,h_k}
\]
and hence
\[
 \|x(t)-y(t)\| \le 2\,h_k\,\Bigl(\sum_{i=1}^{m} K_i\,V_i\Bigr)\,\frac{e^{\Lambda h_k} -1}{\Lambda\,h_k}
\]
\noindent for $t\in [t_k,t_{k+1}]$. Taking the norm of the equations (\ref{firstF},\ref{secondG1a},\ref{secondG1c})
we obtain
\begin{multline*}
 \|x(t_{k+1})-y(t_{k+1})\| \leq
%  \begin{aligned}
    \int_{t_k}^{t_{k+1}} \Bigl( L+\sum_{i=1}^{m} V_iL_i \Bigr)\,\biggl( 2\,h_k\,\Bigl(\sum_{i=1}^{m} K_i\,V_i\Bigr)\,\frac{e^{\Lambda h_k} -1}{\Lambda\,h_k}\biggr) \\
    \qquad\qquad+ \sum_{i=1}^{m} L_i \Bigl( K+\sum_{j=1}^{m} V_jK_j\Bigr) \bigl| \hat{v}_i(t)-\hat{w}_i(t) \bigr|\,dt \\
%  \end{aligned} \\
  \leq {h_k}^2 \biggl( \Bigl( L+\sum_{i=1}^{m} V_iL_i \Bigr)\,\biggl( 2\,\Bigl(\sum_{i=1}^{m} K_i\,V_i\Bigr)\,
              \frac{e^{\Lambda h_k} -1}{\Lambda\,h_k}\biggr)
    + \frac{1}{3}\Bigl(\sum_{i=1}^{m} V_iL_i\Bigr) \Bigl( K+\sum_{j=1}^{m} V_jK_j\Bigr) \biggr) .
\end{multline*}
Using $K'$ and $L'$, we get the desired formula (\ref{constantapproximationsecondordererror}).
\end{proof}

\begin{remark}
Note that as $\Lambda\rightarrow 0$, then $\frac{e^{\Lambda\, h}-1}{\Lambda\,h}\rightarrow 1$.
This is also consistent with Theorem \ref{lnt}. In fact, if $\Lambda=0$, we get
\[
 \|x(t)-y(t)\| \le 2\,h_k\,\Bigl(\sum_{i=1}^{m} K_i\,V_i\Bigr)
\]
\noindent and therefore,
\begin{align}
 \|x(t_{k+1}) - y(t_{k+1})\| \leq h_k^2\,\bigl(\left(K+K'\right)L'/3+2\,K'\,\left(L + L'\right)\bigr),
\end{align}
\noindent which is still of $O(h^2)$. Further, we will not give explicit formulas for the error when $\Lambda=0$.
\end{remark}

\begin{theorem}\label{case2a1}
If all assumptions of Theorem \ref{case2a} are satisfied, and in addition $f(\cdot)$ is $C^2$,
then a formula for calculation of the local error can be given by
\begin{align*}
& \biggl(1-(h_kL/2)\biggr)\|x(t_{k+1}) - y(t_{k+1})\| \le (h_k^2/3)\, \left(3\,K'\,L'\,\,\frac{e^{\Lambda h_k} -1}{\Lambda h_k} + L'\,(K+K')\right)\\
&\qquad\qquad + (h_k^3/4)\, K'\, \biggl(L\,L'+ L^2 + H(K+K')\biggr)\, \frac{e^{\Lambda h_k} -1}{\Lambda h_k}\\
&\qquad\qquad + (11\,h_k^3/24)\,(H\,K'+L\,L')(K+K').
\end{align*}
\end{theorem}

\begin{proof}
The same bounds on functions apply as in Theorem \ref{case2a}.
The formula for $\|x(t_{k+1}) - y(t_{k+1})\|$ is then obtained by taking norms of terms in equations~\eqref{thirdF} and~\eqref{secondG1}.
\end{proof}

\begin{remark}
The computation of the error bound is complicated by that fact that $|v_i(t)-w_i(t)|$ is not uniformly small.
This means that the terms $g(x)(v_i-w_i)$ must be integrated over a complete time step in order to be able to use the fact that
$\int_{t_k}^{t_{k+1}} v_i(t)\,dt = \int_{t_k}^{t_{k+1}} w_i(t)\,dt$, and this must be done \emph{without} first taking norms inside the integral.
As a result, we cannot apply results on the logarithmic norm exactly directly.
Instead, we ``bootstrap'' the procedure by applying a first-order estimate for $\|x(t)-y(t)\|$ valid for any $t\in[t_k,t_{k+1}]$.
\end{remark}

\subsubsection{Local error $O(h^2)+O(h^3)$}
\label{sec:twoparametererror}

 We can attempt to improve the error bounds by allowing $w_i(t)$ to have two independent parameters.
In the general case, we shall see that this gives rise to a local error estimate containing terms of $O(h^2)$ and $O(h^3)$, rather than the anticipated pure $O(h^3)$ error.

We require $w_i(t)$ to satisfy the equations
\be\label{se2}
      \int_{t_k}^{t_{k+1}} v_i(t) - w_i(t) \, dt = 0; \qquad
     \int_{t_k}^{t_{k+1}} (t-t_{k+1/2})\,\,(v_i(t) - w_i(t)) \, dt = 0.
  %\end{gathered}
\ee

If the $w_i$ are taken to be affine functions, $w_i(t)=a_{i,0}+a_{i,1}(t-t_{k+1/2})/h_k$, then we have

\be
 \label{eq:polynomialparameterformulae}
a_{i,0} = \frac{1}{h_k}\, \int_{t_k}^{t_{k+1}} v_i(t)dt; \qquad
a_{i,1} = \frac{12}{h_k^2} \, \int_{t_k}^{t_{k+1}} v_i(t)\,(t-t_{k+1/2}) \, dt.
\ee
\noindent It is easy to see that
\be \label{eq:polynomialparameterbounds}
  |a_{i,0}|\le V_i, \ |a_{i,1}|\le 3\,V_i, \ |w_i(t)| \le 5V_i/2,\ \text{and}\ |\dot w(t)|\le 3V_i/2h_k 
\ee
\noindent and it can further be shown that
\be \label{eq:polynomialquadraticparameterbounds}
  |a_{i,1}|\le 3V_i(1-(a_{i,0}/V_i)^2) .
\ee

An alternative is to use step functions for $w_i$, such as
\begin{displaymath}
   w_i(t) = \left\{
     \begin{array}{rl}
       a_{i,0} & \text{if } t_k\le t < t_{k+1/2}\\
       a_{i,1} & \text{if } t_{k+1/2} \le t \le t_{k+1}.
     \end{array}
   \right.
\end{displaymath}
\noindent Then
\begin{align*}
a_{i,0} &= \frac{1}{h_k}\, \int_{t_k}^{t_{k+1}} v_i(t)\,dt - \frac{4}{h_k^2} \, \int_{t_k}^{t_{k+1}} v_i(t)(t-t_{k+1/2})\,dt\\[3\jot]
a_{i,1} &= \frac{1}{h_k}\, \int_{t_k}^{t_{k+1}} v_i(t)\,dt + \frac{4}{h_k^2} \, \int_{t_k}^{t_{k+1}} v_i(t)(t-t_{k+1/2}) \, dt.
\end{align*}
\noindent Hence
\be
  |a_{i,0}|\le 2\, V_i, \ |a_{i,1}|\le 2\,V_i, \ \text{and}\ |w_i(t)| \le 2\,V_i.
\ee

\medskip

\begin{theorem}\label{case2b}
For any $k\ge 0$, and all $i=1,...,m$, if
\begin{itemize}
 \item $f(\cdot)$ is  $C^2$ vector function,
 \item $g_i(\cdot)$ are non-constant $C^2$ functions,  and
 \item the $w_i$ satisfy~\eqref{se2},
\end{itemize}
\noindent then an error of $O(h^2)$ is obtained.
Moreover, if the $w_i$ are affine functions, $w_i(t)=a_{i,0}+a_{i,1}(t-t_{k+1/2})/h_k$, then a formula for calculation of the error is given by
\begin{align*}
&\left(1-L(h_k/2) - h_k L'\right)\|x(t_{k+1}) - y(t_{k+1})\| \le (h_k^2/4) L'\,\left(11 K + (69/2) K' \right) \\
&\qquad\qquad +(7 h_k^3/8)\,K'\,\left( (4H'+H)\, (K +(5/2)K') + L^2 + \left((9/2) L + 5L' \right)L'\right)
    \frac{e^{\Lambda h_k} -1}{\Lambda h_k}\\
&\qquad\qquad+(7h_k^3/48)  \left( H\,K' + L\,L' \right) \left( K + K'  \right).
\end{align*}
\end{theorem}

\begin{proof}
With the assumptions of the theorem, we can improve the terms (\ref{thirdFd}) and (\ref{thirdGe})
such that they become (\ref{thirdFdp}) and~\eqref{thirdGep}, which are of $O(h^3)$.
In addition to the bounds obtained in~\eqref{eq:polynomialparameterbounds}, we use
\begin{align*}
 \|\dot x(t)\| & \le K + \sum_{i=1}^{m} K_i\,V_i\,=\,K+K'\\
 \|\dot y(t)\| & \le K +  \frac{5}{2} \sum_{i=1}^{m} K_i\,V_i=K+(5/2)K'\\
 \|x(t)-y(t)\| & \le \frac{7h_k}{2} \left(\sum_{i=1}^{m} K_i\,V_i\right)\frac{e^{\Lambda h_k} -1}{\Lambda\,h_k}=\frac{7h_k}{2}\, K'\,\frac{e^{\Lambda h_k} -1}{\Lambda\,h_k}.
\end{align*}

\noindent The formula for the error, $\|x(t_{k+1}) - y(t_{k+1})\| $ with terms (\ref{thirdFdp}) and (\ref{thirdGep})
is then easily obtained. The theorem is proved.
\end{proof}

\medskip

 We now show that with the assumptions of the theorem we cannot in general obtain an error of $O(h^3)$.
Specifically, we assume that $w_i(t)$ are two-parameter polynomial or step functions satisfying
\[ \int_{t_k}^{t_{k+1}} v_i(t)-w_i(t)\,dt = \int_{t_k}^{t_{k+1}} (t-t_{k+1/2})\,(v_i(t)-w_i(t))\,dt = 0 .\]
The following counterexample gives a system for which only $O(h^2)$ local error is possible.

\begin{example}

\noindent Consider the following input-affine system which satisfies assumptions in Theorem~\ref{case2b}:
\begin{equation*}
 \dot{x}_1 = x_2 + v_1 + x_1 v_2; \quad \dot{x}_2 = x_1 + v_2;  \quad  x(t_k)=x_k.
\end{equation*}
Take inputs
\begin{equation*}
v_1(t)=\sin\left(\frac{2\pi}{h_k}(t-t_k)\right), \qquad
v_2(t)=\cos\left(\frac{2\pi}{h_k}(t-t_k)\right).
\end{equation*}

\noindent Using~(\ref{se2}), we get $w_1(t)=-(6/\pi \, h_k)(t-t_{k+1/2})$, $w_2(t)=0$.
Therefore, an approximation equation looks like

\begin{equation*}
 \dot{y}_1 = y_2 + w_1; \quad \dot{y}_2 = y_1
\end{equation*}

As shown in the previous section, the only term which might not have order
$h_k^3$ is the term in~\eqref{thirdGg}
which is reduced to

\[
 \sum_{i=1}^2 \int_{t_k}^{t_{k+1}} Dg_2(x(t))g_i(x(t))\,v_i(t) \hat{v}_2(t) dt,
\]

\noindent since $Dg_1=0$. When $i=2$, we have $\frac{1}{2} \frac{d}{dt}(\hat{v}_i^2(t))= v_i(t)\hat{v}_i(t)$,
and hence we can integrate by parts once more to get the $O(h^3)$. Then we are left with
\begin{align*}
\int_{t_k}^{t_{k+1}} Dg_2(x(t))g_1(x(t))\,v_1(t) \hat{v}_2(t) dt = -\frac{h_k^2}{4\pi}\,\, [1\,\,\, 0]^T ,
\end{align*}
\noindent a term of $O(h^2)$.
\end{example}

\subsubsection{Local error of $O(h^3)$}
\label{sec:twoparameteradditiveinputerror}

 We showed that for a general input-affine system, a local error of order
$O(h^3)$ cannot be obtained using affine approximate inputs $w(a,t)$. However, if in addition, we assume that $g_i(\cdot)$
are constant functions or we have a single input then we can obtain a local error of $O(h^3)$.
If $g_i(\cdot)$ are constant functions, then the error calculation is equivalent to the error calculation
of an even simpler case, so called additive noise case. The equation is then given by
\be\label{an}
\dot x(t) = f(x(t)) + v(t).
\ee
\noindent Here, $v(t)=(v_1(t),...,v_n(t))$ is vector-valued.

\begin{corollary}\label{case3a}
 For any $k\ge 0$,
\begin{itemize}
 \item if the system has additive noise,
 \item $f(\cdot)$ is a $C^2$ function, and
 \item $w_i(t)$ are real valued functions defined on $[t_k,t_{k+1}]$ which satisfy equations~\eqref{se2},
\end{itemize}
\noindent then an error of $O(h^3)$ is obtained. Moreover, for $w_i(t)=a_{i,0}+a_{i,1}(t-t_{k+1/2})/h_k$, the formula for the local error is given by:
\be\label{ine}
\begin{aligned}
\bigl( 1-(h_k/2) L \bigr)\|x(t_{k+1})-y(t_{k+1})\| &\le \frac{7}{48}\, h_k^3\,K'\,H\,(K+K')\\
 &\,\,+ \frac{7}{8}\,h_k^3\,K'\,\Bigl(L^2\, +\, H\,(K+5K'/2)\Bigr)\frac{e^{\Lambda h_k}-1}{\Lambda\,h_k}.
\end{aligned}
\ee
\end{corollary}

\noindent The formula for the error in additive noise case is simplified
because $L'=H'=0$. If we write $||v(t)||= K'$, then the result follows directly from Theorem~\ref{case2b}.

\begin{corollary}\label{case3b}
For any $k\ge 0$, if
\begin{itemize}
\item the input-affine system has single input, i.e., $m=1$ in~\eqref{ca}
\item $f(\cdot)$ and $g(\cdot)$ are $C^2$ functions, and
 \item $w(t)$ is a real valued function defined on $[t_k,t_{k+1}]$ which satisfies equations~\eqref{se2},
\end{itemize}
\noindent then an error of $O(h^3)$ is obtained. Moreover, for $w(t)=a_{0}+a_{1}(t-t_{k+1/2})$, the formula for
the local error is given by
\begin{align*}
 &\left( 1-(h_k/2) L - h_kL'\right)\|x(t_{k+1})-y(t_{k+1})\| \le \\
 &\frac{7\,h_k^3}{8}\,K' \left((H\,+\,10\,H')(K+(5/2)K') + L^2\, +\,(25/2)\,L\,L'\,+\,25\,(L')^2\right)\frac{e^{\Lambda h_k}-1}{\Lambda\,h_k}\\
 &+ \frac{h_k^3}{48}\,(K+K')\,\left((7/6)(H\,K'+L\,L') + 28\,(H'\,K+L\,L') + 29\,(H'\,K'+(L')^2) \right).
\end{align*}
\end{corollary}
\begin{proof}
The result follows since the only term which is not $O(h^3)$ in~(\ref{thirdF},\ref{thirdG}) is~\eqref{thirdGg}.
In the one-input case, this simplifies to
\[ \int_{t_k}^{t_{k+1}} Dg(x(t))\,g(x(t))\,\bigl(\hat{v}(t)\,v(t)-\hat{w}(t)\,w(t)\bigr)\,dt .\]
However, we can integrate by parts to obtain
\begin{align*}
 \eqref{thirdGg} &= \Bigl[ Dg(x(t))\,g(x(t))\,\bigl(\hat{v}(t)^2-\hat{w}(t)^2\bigr) \Bigr]_{t_k}^{t_{k+1}} \\
  &\qquad\qquad - \int_{t_k}^{t_{k+1}} D\bigl(Dg(x(t))\,g(x(t))\bigr)\,\dot{x}(t)\,\bigl(\hat{v}(t)^2-\hat{w}(t)^2\bigr)\,dt .
\end{align*}
The first term vanishes since $\hat{v}(t_{k_1})=\hat{w}(t_{k+1})$, and the second is $O(h^3)$ since $\hat{v}(t)$ and $\hat{w}(t)$ are $O(h)$.
Taking all the bounds as in Theorem \ref{case2b}, the formula is easily obtained.
\end{proof}
 Observing the error given by equations~\eqref{thirdF} and~\eqref{thirdG} , we see that if in addition to
satisfying equations given in~\eqref{se2}, the functions $w_i(\cdot)$ also satisfy
\be\label{a3}
 \int_{t_k}^{t_{k+1}} v_i(t)\hat{v}_j(t) - w_i(t)\hat{w}_j(t)\,\,dt\, =\, 0.
\ee
\noindent then we could get an error of $O(h^3)$. The question remains as to whether we can find
functions $w_i(\cdot)$ that satisfy the conditions~(\ref{se2},\ref{a3}).
Since the functions $w_i(\cdot)$ cannot be computed independently any more,
the number of parameters of each $w_i(\cdot)$ will depend on the number of inputs.

\begin{theorem}
 For any $k\ge 0$, if
\begin{itemize}
 \item $f(\cdot)$, $g_i(\cdot)$ are $C^2$ real vector functions, and
 \item $w_i(a_{i,0},...,a_{i,p\!-\!1},t)$ are real valued, defined on $[t_k,t_{k+1}]$, and satisfy
\be\label{se3}
\begin{gathered}
 \int_{t_k}^{t_{k+1}} v_i(t) - w_i(t) \, dt = 0\\
 \int_{t_k}^{t_{k+1}} (t-t_{k+1/2})\,(v_i(t) - w_i(t)) \, dt = 0\\
 \int_{t_k}^{t_{k+1}} v_i(t)\hat{v}_j(t) - w_i(t)\hat{w}_j(t)\,\,dt\, =\, 0,
\end{gathered}
\ee
\end{itemize}
\noindent for all $i,j=1,...,m$, then an error of $O(h^3)$ can be obtained.
Note that it suffices to take $j<i$ in~\eqref{se3}, and that the number of parameters $p$ in each $w_i$ must satisfy $p \geq (m+3)/2$.
Taking polynomials of minimal degree $d$, we obtain $d=\lceil(m+1)/2\rceil$.
\end{theorem}

\begin{proof} If we can find $w_i(t)$ that satisfies above, then it is obvious that the only remaining $O(h^2)$ term~\eqref{thirdGg}
can be integrated by parts once more in order to give a term of $O(h^3)$.
This follows from Theorem \ref{case2}, Corollary~\ref{case3b} and the formulae in Section~\ref{errde}.

To see that we can find the desired functions $w_i(\cdot)$, we consider polynomial approximations $w_i$ of degree $d=p+1$.
We will show that it is possible to solve for the parameters of $w_i$'s.
If $m=1$, see Corollary \ref{case3b}. The system of equations~\eqref{se3} consists of at most $m+m+m(m-1)/2=m(m+3)/2$ independent equations.
To see that third equation in~\eqref{se3} has at most $m(m-1)/2$ independent equations necessary to be zero, notice that when $i=j$ we have
\begin{align*}
 \int_{t_k}^{t_{k+1}} v_i(t) \hat{v}_i(t) - w_i(t) \hat{w}_i(t)\, dt&= (1/2) [\hat{v}_i^2(t_{k+1})- \hat{w}_i^2(t_{k+1})],
\end{align*}
\noindent and therefore we can integrate by parts once more to get error of $O(h^3)$. When $j>i$ integration by parts gives
\begin{align*}
 \int_{t_k}^{t_{k+1}} v_i(t) \hat{v}_j(t) - w_i(t) \hat{w}_j(t)\, dt &= \bigl[\hat{v}_i(t)\,\hat{v}_j(t) - \hat{w}_i(t)\hat{w}_j(t)\bigr]_{t_k}^{t_{k+1}}\\
&\qquad\qquad - \int_{t_k}^{t_{k+1}} \hat v_i(t) {v}_j(t) -\hat w_i(t) {w}_j(t)\, dt
\end{align*}
and the first term vanishes since $\hat{v}_i(t_{k+1}) = \hat{w}_i(t_{k+1})$.
The number of parameters that each $w_i(\cdot)$ has is $p=d+1$.
Thus, in total, we have $mp$ parameters.
In order to guarantee that we can solve all the equations for the $w_i(\cdot)$'s,
we need that $mp \ge m(m+3)/2$. This implies that $p\ge (m+3)/2$.
Taking polynomials of minimal degree, we see that we require $d=\lceil(m+1)/2\rceil$.
\end{proof}
In what follows, we write $C(n,m)=n!/(m!\,(n-m)!)$, the formula for combinations
(selecting $m$ elements among $n$ elements).

\begin{table}[t!]
\begin{center}
\begin{tabular}{| c | c | c | c |}
  \hline
  \#Inputs & \#Equations & Degree & \#Parameters\\
  $ m $ & $m(m+3)/2$ & $d$ & $m(d+1)$ \\
  \hline
  1 & 2 & 1 & 2\\
  \hline
  2 & 5 & 2 & 6\\
  \hline
  3 & 9 & 2 & 9\\
  \hline
  4 & 14 & 3 & 16\\
  \hline
  5 & 20 & 3 & 20\\
  \hline
  6 & 27 & 4 & 30 \\
  \hline
  10 & 65 & 6 & 70 \\
  \hline
\end{tabular}
\end{center}
\caption{The number of independet equations which need to be solved, the minimal degree of a polynomial $w_i(\cdot)$ required, and number of available parameters in order to obtain $O(h^3)$ local error for $m$ inputs.}
\label{h3}
\end{table}

 In Table \ref{h3}, we present the degree of $w_i(\cdot)$ needed for one to obtain $O(h^3)$ for different number of inputs.
In addition, the number of equations involved and the number of independent parameters in $m$ functions that have to be found are given.

\subsubsection{Higher Order Local Error}

 It is possible to generalize the approach used to generate $O(h^3)$ local error.
 With additional smoothness requirements on the functions $f(\cdot)$ and $g_i(\cdot)$'s, we can get even higher-order local errors.
  In order to simplify the notation, we set $g_0=f$ and $v_0=1$. Then the input-affine system~\eqref{ca}
becomes
\[
 \dot x(t) = \sum_{i=0}^{m} g_i(x(t))v_i(t) .
\]

\noindent Let $g_i\in C^r$ for all $i=0,...,m$, and denote by
\[
 \dot y(t) = \sum_{i=0}^{m} g_i(y(t))w_i(a_i,t)
\]
\noindent the corresponding approximate system. The local error of $O(h^{r+1})$ can be obtained if $w_i(a_i,t)$ is finitely parametrised, $a_i=(a_{i,0},...,a_{i,d})$ with $d$ being sufficiently large, and satisfying

\begin{subequations}
\begin{align}
 \int_{t_k}^{t_{k+1}}v_i(t)\,dt &=\int_{t_k}^{t_{k+1}} w_i(t)\,dt \label{e1} \\
% \int_{t_k}^{t_{k+1}} v_j(t) \hat{v}_i(t)\,dt &= \int_{t_k}^{t_{k+1}} w_j(t) \hat{w}_i(t)dt,\,\,i,j\ge 0,\,\,j<i\label{e2} \\
 \int_{t_k}^{t_{k+1}} v_j(t) \int_{t_k}^{t} v_i(s)\,ds \ dt &= \int_{t_k}^{t_{k+1}} w_j(t) \int_{t_k}^{t} w_i(s)ds\ dt \label{e2} \\
% \int_{t_k}^{t_{k+1}} v_k(t) \hat{v}_j(t) \hat{v}_i(t)\,dt &=  \int_{t_k}^{t_{k+1}} w_k(t) \hat{w}_j(t) \hat{w}_i(t)\,dt\,\,i,j,k\ge 0,\,\,j\le i\label{e3}\\
  \int_{t_k}^{t_{k+1}} v_k(t) \int_{t_k}^{t} v_j(s) \int_{t_k}^{s}v_i(r)\,dr\;ds\ dt &=  \int_{t_k}^{t_{k+1}} w_k(t) \int_{t_k}^{t} w_j(s) \int_{t_k}^{s}w_i(r)\,dr\;ds\ dt \label{e3}
\end{align}
\begin{multline}
\int_{t_k}^{t_{k+1}} v_{i_r}(s_r) \int_{t_k}^{s_r} v_{i_{r-1}}(s_{r-1})\cdots \int_{t_k}^{s_2} v_{i_1}(s_1)\,ds_1\,\cdots\,ds_{r-1}\,ds_r = \\ \qquad\qquad \int_{t_k}^{t_{k+1}} w_{i_r}(s_r) \int_{t_k}^{s_r} w_{i_{r-1}}(s_{r-1})\cdots \int_{t_k}^{s_2} w_{i_1}(s_1)\,ds_1\,\cdots\,ds_{r-1}\,ds_r
\end{multline}
\end{subequations}

\noindent We can restrict to $i\geq1$ in~\eqref{e1}.
In~\eqref{e2} we can restrict to $i\geq j+1$ as explained in previous subsection. %
In~\eqref{e3}, we can simplify to
\begin{equation*}
 \int_{t_k}^{t_{k+1}} v_k(t) \hat{v}_j(t) \hat{v}_i(t)\,dt =  \int_{t_k}^{t_{k+1}} w_k(t) \hat{w}_j(t) \hat{w}_i(t)\,dt; \qquad i,j,k\ge 0,\ j\le i
\end{equation*}

 Note that for the first two equalities above we need $m + C(m+1,2)$ equations, where $C(n,m)=n!/m!(n-m)!$, which in total gives $(m/2)(m^2+4m+7)$.
For the third one, we need additional $m + 3\,C(m+2,3)$.
In general, it is not easy to see the formula for the number of equations.
The number of parameters and the required degree for $O(h^4)$ are given by
$(m/2)(m^2+4m+7)$ and $N=\lceil(1/2)(m^2 + 4m + 5)\rceil$ respectively.

\section{Improvements and Generalizations}\label{GDI}

In this section, we consider techniques for improving the estimates obtained, and for generalizing the methods to differential inclusions with constraints.

\subsection{Improved approximate solution sets}

The previous error estimates were based on bounding the parameters appearing in the form of the input $w(t)$.
For example, supposing a single input $v(t)\in[-1,+1]$ and taking $w(t)=a_0 + a_1 (t-t_{k+1/2})/h_k$ satisfying $\int_{t_k}^{t_{k+1}} v(t)-w(t)\, dt = \int_{t_k}^{t_{k+1}} t\,v(t)-w(t)\, dt = 0$, we find $|a_0|\leq 1$ and $|a_1|\leq 3$.
However, if $a_0=\pm1$, then $v(t)\equiv \pm1$ on $[t_k,t_{k+1}]$, so $a_1=0$. Similarly, if $|a_1|=3$ then $a_0=0$.

For a given $a_0$, we can maximise $a_1$ by taking
\[ v(t) = \begin{cases} -1 \text{ for } t_k \leq t \leq t_k+\alpha h_k, \\
                        +1 \text{ for } t_k+\alpha h_k \leq t \leq t_k+h_k = t_{k+1}. \end{cases} \]
where $\alpha = (1-a_0)/2$. For this $v$, we find
\[ \begin{aligned}
     a_1 &= \frac{12}{h_k^2} \int_{t_k}^{t_{k+1}} (t-t_{k+1/2}) \, v(t)\,dt
        \ = \ \frac{12}{h_k^2} \biggl( \int_{\alpha h_k}^{h_k} (t-h_k/2)\,dt -  \int_{0}^{\alpha h_k} (t-h_k/2)\,dt \biggr) \\
         &= 3\bigl(1-(1-2\alpha)^2\bigr) \ =\ 3(1-a_0^2)
 \end{aligned} \]
yielding the constraint
\[ a_0^2 + |a_1| / 3 \leq 1 . \]
We can therefore set
\begin{equation}\label{eqn:reducedparameterdomain}
  w_k(t) = a_0 + { 3(1-a_0^2)b_1  }\,(t-t_{k+1})/h_k \quad \text{with} \quad a_0,b_1\in[-1,+1] .
\end{equation}
This will yield sharper estimates than~\eqref{eq:polynomialparameterformulae}.

\subsection{Differential inclusions with constraints}

Up to now, we have considered affine differential inclusions of the form
\[ \dot{x}(t) = f(x(t)) + \sum_{i=1}^{m} g_i(x(t)) v_i(t) \text{ with } v_i\in[-V_i,+V_i] . \]
In other words, the disturbances $(v_1,\ldots,v_m)$ lie in a coordinate-aligned box $[-V_1,+V_1]\times\cdots\times[-V_k,+V_m]$.
In many problems, the set $V$ containing $(v_1,\ldots,v_m)$ will not be box, but some more complicated set.
We could use our method directly to compute over-approximations to the solution set by taking an over-approximating bounding box $\widehat{V}$ to $V$, but this will typically yield extra solutions, even in the limit of small step size.
Instead, we seek to restrict solutions to those of the original system.

The right-hand-side of the differential inclusion is convex if, and only if, $V$ is a convex set, so it suffices to restrict to this case.
We can write
\[ V = \{ (v_1,\ldots,v_m) \mid v_i \in [-V_i,+V_i] \wedge c(v_1,\ldots,v_k)\leq 0 \} \]
where $c:\mathbb{R}^m\rightarrow\mathbb{R}$ is a convex function.
(More generally, we could consider the disjunction of several such constraints.)
The constraint $c$ yields restrictions on the form of the $w_i$.
For second-order estimates using
\[ w_{k,i}(t) = a_{k,i} = \frac{1}{h_k} \int_{t_k}^{t_{k+1}} v_i(t)\,dt \]
we simply need to introducte the constraints
\be \label{eqn:constraints} c(a_{k,1},\ldots,a_{k,m}) \leq 0 \ee
at every step.
For higher-order estimates, the relationship between the parameters and the constraint function may be more complicated; in particular, it need not be the case that $c(w_{k,1}(t),\ldots,w_{k,m}(t))\leq0$ holds.

\subsection{Pseudo-affine inputs}

In this section, we consider differential inclusions of the form

\be\label{ial}
 \dot x(t) = g(x(t)) + G(x(t))q(v(t)),\,\,\,x(0)=x_0,\,\,\, v(t)\in V
\ee

\noindent where $V$ is compact, convex subset of $\mathbb{R}^m$, and $g:\mathbb{R}^n\rightarrow \mathbb{R}^n$,
$G:\mathbb{R}^n \rightarrow \mathbb{R}^{n \times p}$, and $q:\mathbb{R}^m \rightarrow \mathbb{R}^p$.
The inclusion above can be viewed in two different ways.

 One way is to consider
the right-hand side as a function which is non-linear in the input. For example,
consider a one-dimensional polynomial system with inputs,
\[
\dot x(t) = x^7\, v_1^2 + x\, v_2^2 + x^3\,v_1\,v_2 + x^5,\,\,\, (v_1,v_2)\in V \subset \mathbb{R}^2.
\]

\noindent This has a form $g(x) + G(x)q(v)$ by taking
$g(x)=x^5$, $G(x)=(x^7,\, x,\,x^3)$, and $q(v) = (q_1(v),q_2(v),q_3(v)) = (v_1^2,\,v_2^2,\, v_1\,v_2)$.

 The other way is to
consider the right-hand side as a function which is linear in the input

\[
 \dot x(t) \in g(x(t)) + G(x(t))r(t),\,\,\,r(t)\in q(V)
\]

\noindent This corresponds to the case where, in general $V$, i.e., $q(V)$ above, is convex, but not necessary a box as it was assumed in the
previous section. For example, we can consider $V$ given by constraints such as
$V=\{v(t)\,|\,c(x(t),v(t))\le 0\}$ or $V=\{v(t)\,|\,e(x(t),v(t))=0\}$ for some
continuous functions $c(\cdot)$ and $e(\cdot)$.

 In order to compute reachable sets of the system~\eqref{ial}, we proceed as in the previous section.
First we construct an ``approximate'' system

\[
 \dot y(t) = g(y(t)) + G(y(t))w(t),
\]

\noindent and then get an error on the approximation.
The local error will be essentially obtained in the same way as before,
i.e., Theorems \ref{case1}-\ref{case3b}, but with certain additional assumptions.
To see what the assumptions should be, suppose that we want to get an error as in Theorem \ref{case2b}.
Then $w(t)=(w_1(t),...,w_m(t))$ is affine and satisfies
the integral equalities

\begin{gather*}
      \int_{t_k}^{t_{k+1}} q(v(t)) - w(t) \, dt = 0 \\
      \int_{t_k}^{t_{k+1}} t\,(q(v(t))-w(t)) \, dt=0.
\end{gather*}
\noindent As before, we get
\begin{gather*}
      a =(a_1,...,a_m)= \frac{1}{h} \int_{t_k}^{t_{k+1}} q(v(t)) dt\\
      b =(b_1,...,b_m)= \frac{12}{h^3} \int_{t_k}^{t_{k+1}} q(v(t))(t - t_{k+1/2}) dt
\end{gather*}

\noindent Obviously, we can take box over-approximations for $a$ and $b$, and obtain
over approximations of the reachable sets. However, if $q$ is nonlinear, or
$V$ is not a box, but some general convex set, then box over-approximations
for $a$ and $b$ could result in large over-approximation of the reachable sets.
Therefore, if the set $q(V)$ satisfies additonal assumptions, we can get optimal results
for the parameters $a$ and $b$. For example, if
$q(V)$ is a convex set, centered around the origin, we get $a\in q(V)$
and $b\in (3/h)q(V)$, which gives optimal bounds for the coefficients $a$ and $b$.

\section{Numerical Results}\label{num}

We now illustrate the use of our algorithm by computing reachable sets for some simple systems.

\medskip
\subsection{Van Der Pol Oscillator}\label{VDP}
%%%%order of numerical method
We consider perturbed Van der Pol oscillator given by
\begin{align*}
 & \dot x= y\\
 & \dot y= -x + 2\, (1-x^2)\,y + v,
\end{align*}
\noindent where $v$ represents additive noise. 
We use the method described in Section~\ref{sec:twoparametererror} and the error bound~\eqref{ine} for additive inputs.
If we take $D=[0,2]\times[-1,3]$ to be the region of computation, then
we get $K=20$, $L=31$, $\Lambda=27$, and $H=12$. In addition, if we assume that
$v(\cdot)\in [-0.08,0.08]$, i.e., $A=0.08$, we obtain
\[
\epsilon=\|x(t_{k+1})-y(t_{k+1})\|\le 11.24\,h^3 + 168.17\, h^3\,\frac{e^{27\,h}-1}{27h}
\]
We use the algorithm described in Section~\ref{Algo} to compute
the solution set
for the set of initial points $X_0=[0.1,0.105]\times[1.5,1.505]$ over the time interval
$[0,1.5]$. Because the bounds $K,\,L,\,\Lambda$, and $H$ are rather large,
we use fairly small step size, $h=0.001$, yielding an analytical single-step error
of $\epsilon = 1.817092608\times 10^{-7}$.
In Figures~\ref{pvdp1} and \ref{pvdp2} we show solution set of the perturbed
Van der Pol oscillator using the above values.
In figure \ref{pvdp1}, splitting of the domain was performed at $t_1=0.6$ and $t_2=1.2$.
At $t_1$ the set was divided in half along $x$-axis, and at $t_2$
the set was divided in half along $y$-axis.
The computed reachable set after $T=1.5$ is a union of the following four sets:
\begin{align*} R(X_0,T) \ \subset \ &[1.46104,1.66704]\times [-0.482307,-0.272922] \\ &\qquad \;\cup\; [1.60834,1.80823]\times[-0.438931,-0.263936] \\ &\qquad\qquad \;\cup\; [1.50247,1.70832]\times[-0.466819,-0.269152] \\ &\qquad\qquad\qquad \;\cup\; [1.65202,1.8518]\times[-0.424135,-0.259941] . \end{align*}
Moreover, if there was no splitting performed the reachable set at $T=1.5$ is then
\[ R(X_0,T)\subset [1.43018,1.88571]\times[-0.513789,-0.197579] , \] 
and the computed solution set is presented in \ref{table:pho2}. 
From the results obtained, it turns out that the reachable set was smaller when splitting was performed. 

Note that the set $D$ in this case was chosen approximately, so that for initial condition $X_0$ and time of computation $T=1.5$,
the solution set of the differential inclusion stays inside $D$. This is done so that analytical error does not have to be recomputed
at each time step. In general, it is not necessary to know {\it a-priori} the region of computation.
In fact, at each time step, we can check  whether the reachable set is inside $D$, if not, we can choose new $D$ and recompute the error accordingly.

\bigskip

\begin{figure}[h]
\centering
  \begin{minipage}[t]{0.45\linewidth}
  \includegraphics[width=2.5in]{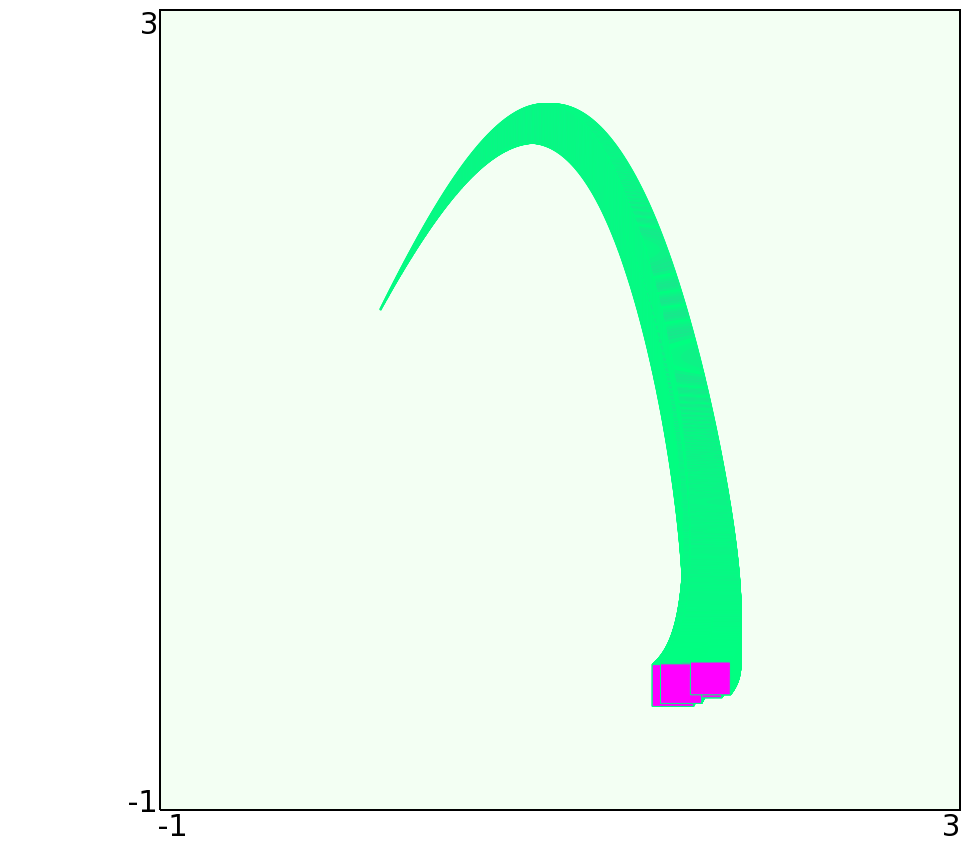}
  \caption{Evolution of the Perturbed Van der Pol Oscillator: splitting performed at $t_1=0.6$ and $t_2=1.2$.}
  \label{pvdp1}
  \end{minipage}%
  \hspace{0.5in}%
  \begin{minipage}[t]{0.45\linewidth}
   %\centering
     \includegraphics[width=2.5in]{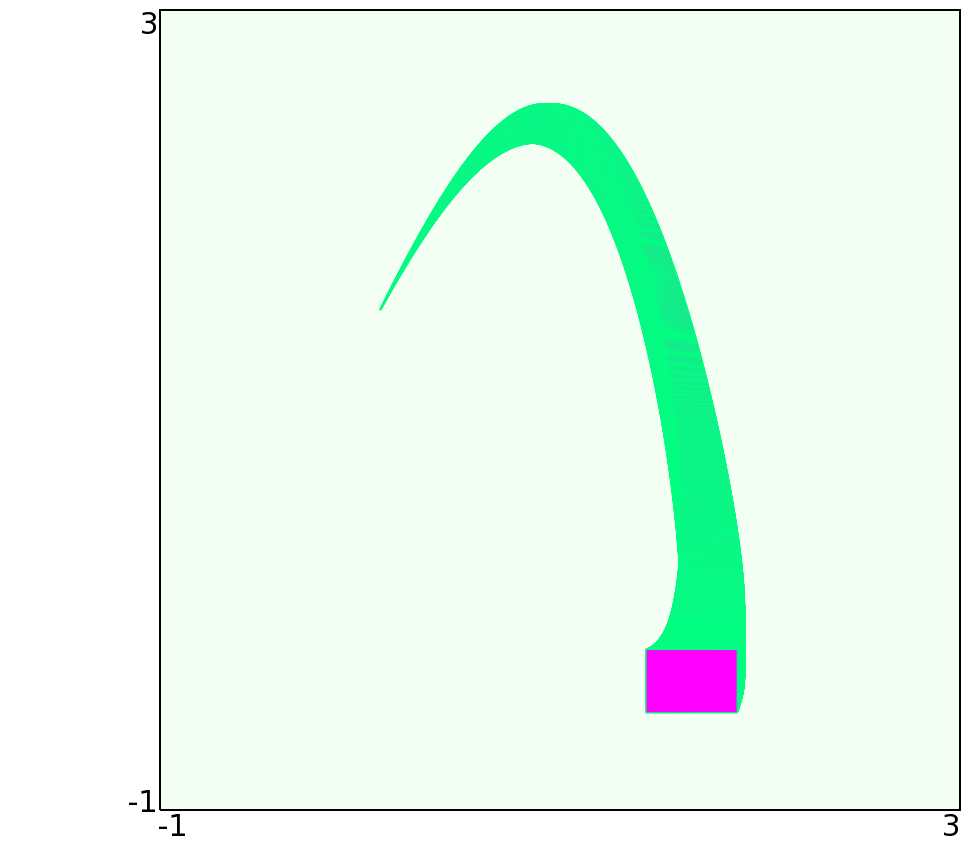}
     \caption{Evolution of the Perturbed Van der Pol Oscillator: no splitting performed.}
  \label{pvdp2}
  \end{minipage}
\end{figure}

 Figures~\ref{pvdp1} and~\ref{pvdp2} show that our method is effective in practice for computing rigorous over-approximations
of the solution sets of nonlinear differential inclusions. To prove this, we compare the results
of computation of the algorithm presented here with
the ones given in \cite{KZ}.

\subsection{Perturbed Harmonic Oscillator}\label{PHO}

The equations for the perturbed harmonic oscillator are given by
\begin{align*}
 & \dot x= y + v_1\\
 & \dot y= -x + v_2,
\end{align*}
\noindent where $v_i$'s represent bounded noise.  Suppose that the range of $v_1$ and $v_2$
is $[-A_1,A_1]$ and $[-A_2,A_2]$ respectively. Notice that noise is additive, and
therefore we can use formula \eqref{ine} to compute the (analytical) error.  
In terms of our general set up we have $f(x,y)=(y,-x)$, $g_i=1$, for $i=1,2$. 
Hence, we get $\Lambda=1$, $L=1$, $H=0$, and $K'=A_1+A_2$. The one step time error is then given by the 
following formula  
\[
\epsilon=\frac{7\,h^3\,}{4(2-h)}\,\frac{e^h-1}{h}\,\max\{A_1, A_2\}.
\]
%
% We use a version of the algorithm described in \ref{Algo} to compute
% the solution set of the perturbed harmonic oscillator.
% In particular, at each time step, we compute the flow function
% (step $(i)$). We skip step $(ii)$, and we evaluate
% the flow function at the value of parameters.
For comparison purposes, Table \ref{table:pho1} is equivalent to a table given in \cite{KZ}.
The total time of computation is $T=2\pi$, $A_1=0$, and initial condition is the box $(1,0)+[-\delta,\delta]^2$.
Note that diameter of the set $[a_1,a_2]\times [b_1,b_2]\in \mathbb{R}^2$
is $\max \{a_2-a_1, b_2-b_1\}$, and radius of the set is half of diameter.

\noindent From Table \ref{table:pho1}, one can see that in most cases our results are better then one obtained
in \cite{KZ}. 
In case~\ref{case:largestep}, the time step is $h=2\pi/9=0.698131$, for which the analytical error $\epsilon=0.066170$ is too large to hope for sharp results.
In case~\ref{case:smallstep}, handling the large number of time steps requires more sophisticated techniques for simplifying the representation of the intermediate sets than are currently used in our code, and this is the major contribution to the error.

\newcounter{testcase}
\begin{table}[ht]
\caption{Perturbed Harmonic Oscillator $T=2\pi$}   % title of Table
\centering 
\vspace{0.1cm}                         % used for centering table
\begin{tabular}{c |  c | c | c | c | c }            % centered columns (4 columns)
\hline\hline                     %inserts double horizontal lines
case & $A_2$ & $\delta$ & num. of steps & Our Diameter & Diameter in \cite{KZ}\\ [0.3ex]   % inserts table
%heading
\hline                              % inserts single horizontal line
\refstepcounter{testcase}\arabic{testcase}\label{case:largestep} & 0.1 & 0.01 & 9 & {\bf 3.91258} & 1.178825 \\               % inserting body of the table
\refstepcounter{testcase}\arabic{testcase} & 0.1 & 0.01 & 100 & {\bf 0.8382630} & 0.8453958 \\
\refstepcounter{testcase}\arabic{testcase}\label{case:smallstep} & 0.1 & 0.01 & 1000 & {\bf 65.4376} & 0.8225159 \\
%\refstepcounter{testcase}\arabic{testcase} & 0.1 & 0.01 & 10000 & ??? & 0.8202514 \\
%\refstepcounter{testcase}\arabic{testcase} & 0.1 & 0.01 & 100000 & ??? & 0.8200251 \\          % [1ex] adds vertical space
\hline     
\refstepcounter{testcase}\arabic{testcase} & 0.1 & 0 & 100 & {\bf 0.8186080} & 0.8253958 \\               % inserting body of the table
\refstepcounter{testcase}\arabic{testcase} & 0.1 & 0.01 & 100 & {\bf 0.8382630} & 0.8453958 \\
\refstepcounter{testcase}\arabic{testcase} & 0.1 & 0.1 & 100 & {\bf 1.018708} & 1.025396 \\
\hline 
\refstepcounter{testcase}\arabic{testcase} & 0.01 & 0.01 & 100 & {\bf 0.1018380} & 0.1025396 \\
\refstepcounter{testcase}\arabic{testcase} & 0.1 & 0.01 & 100 & {\bf 0.8382630} & 0.8453958 \\   
\refstepcounter{testcase}\arabic{testcase} & 1 & 0.01 & 100 & {\bf 8.205280}  & 8.273958 \\
%\refstepcounter{testcase}\arabic{testcase} & 10 & 0.01 & 100 & ??? & 82.55958 \\ 
\hline                         %inserts single line
\end{tabular}
\label{table:pho1}          % is used to refer this table in the text
\end{table}

 When both $A_1$ and $A_2$ are nonzero, i.e. $A_1=A_2=0.1$, our results and results from
\cite{KZ} are given in Table \ref{table:pho2}. 
Here, we present results only for smaller time steps, even though in \cite{KZ} the results were given for time steps up to $h=0.799$. 
We give both second-order and third-order local error estimates.
%Again, for us it does not make sense to take time steps large enough because the analytical error is too large. 
We can see from Table \ref{table:pho2} that for $h=0.25$ we are starting to get significantly worse results then in \cite{KZ}, but for smaller time steps the results are comparable. 
Here, the total time of computation is $T=h$ (one time step), and $\delta=0$.

\begin{table}[htb]
\caption{Perturbed Harmonic Oscillator $T=h$}   % title of Table
\centering 
\vspace{0.1cm}                         % used for centering table
\begin{tabular}{c |  c | c | c | c }            % centered columns (4 columns)
\hline\hline                     %inserts double horizontal lines
case & h & Our Radius(2) &  Our Radius(3) & Radius in \cite{KZ}\\ [0.3ex]   % inserts table
%heading
\hline                              % inserts single horizontal line
%? & {\bf 0.280646} & ? \\
1 & 0.25 & {\bf 0.0420586} & {\bf 0.0313667}  & 0.0284025 \\           % inserting body of the table
2 & 0.1 & {\bf 0.0125864} & {\bf 0.0108419} & 0.0105171 \\
3 & 0.01 & {\bf 0.00102509} & {\bf 0.00100759} & 0.00100502 \\
4 & 0.001 & {\bf 0.00010026} & {\bf 0.00010009} & 0.00010005 \\
%1 & 0.25 & {\bf 0.0313759}  & 0.0284025 \\           % inserting body of the table
%2 & 0.1 & {\bf 0.0108439} & 0.0105171 \\
%3 & 0.01 & {\bf 0.00100900} & 0.00100502 \\
%4 & 0.001 & {\bf 0.00010050} & 0.00010005 \\
\hline                         %inserts single line
\end{tabular}
\label{table:pho2}          % is used to refer this table in the text
\end{table}

We see that the radius of the enclosure is dominated by the growth due to the noise in the differential inclusion.
The reason why our third-order error estimates give worse enclosures than those of~\cite{KZ} is unclear; however we note that the error estimates obtained there were computed exactly by hand, and our automated methods are better than those of~\cite{KZ} based on the logarithmic norm. Moreover, in~\cite{KZ} they use the 2-norm for the logarithmic norm which gives better results for this example.

%\greencomment{I tried estimating the reachable sets using the 2nd order algorithm with uniform inputs, and obtained comparable results but faster. Maybe we should mention this. In general, I think that for the step sizes we use, the higher order algorithms introduce too many parameters; higher order is probably more useful for ``reasonable'' estimates computed quickly. However, the coefficient in the third-order error for the harmonic oscillator is lower.}
%\redcomment{I am not sure what results did you get exactly, but being that the results are so close in here, I think we should have a strong argument at the end of this section or at the concluding remarks section to why is our algorithm important}

\subsection{Rossler Equations}\label{RE}

The Rossler equations are given by
\begin{align*}
\dot x&=-(y+z) + v_1\\
\dot y&=x+0.2y + v_2\\
\dot z&=0.2+z(x-a) + v_3
\end{align*}
\noindent We aim to estimate the image of the initial set 
\[ X_0= \{0\} \times [-10.3\times 10^{-4},+10.3\times10^{-4}] \times [-0.03\times 10^{-4},+0.03\times10^{-4}]  \]
under the return map $P$ to the Poincar\'e  section $\Sigma=\{x=0, \,\dot{x}>0\}$ 
for the parameter value $a=5.7$ and noise $v_i \in [-10^{-4},10^{-4}]$ for $i=1,2,3$. 
Rather than compute the crossing time for each trajectory, we computed a time interval $T$ containing the first crossing time by comparing the sign of $x$ over the sets $R_k$, and used the estimate $\{0\}\times P(X_0) \subset R(X_0,T)$.

With time step $h=0.005$, total time $T=11.1$,  and region of computation $D=([-25,25],[-25,25],[-25,35])$, 
we obtain an analytical error of $e=8.586\cdot 10^{-8}$ and 
\[
R(X_0,T)=([-0.15572,0.15391],[-3.75926,-3.41772],[0.03139,0.03398]).
\]
\noindent In \cite{KZ}, $R(X_0,T)=([-0.211150,0.20888],[-3.69781,-3.47352],[0.03117,0.03327])$. (They did not specify the time step or the total time it took to compute the value of the poincare map $R(X_0,T)$.)
In this case neither of the sets is better then other, but they are comparable, and
hence we show that our algorithm can also provide good estimates when computing
over rather difficult regions, see \cite{KZ}.

\section{Concluding Remarks}\label{Disc}

 In this paper, we have given a numerical method for computing rigorous over-approximations of the reachable sets
of differential inclusions. The method gives high-order error bounds for single-step approximations, which is an improvement
of the first-order methods previously available. By providing improved control of local errors, the method allows for accurate computation of reachable sets over longer time intervals.

 We give several theorems for obtaining local errors of different orders.
It is easy to see that higher order errors (improved accuracy) require approximations
that have larger number of parameters (reduced efficiency). The growth of the number of parameters is an issue, in general. Sophisticated methods for handling this are at least as important as the single-step method.
The question remains as to approximate solution (Theorems \ref{case1}-\ref{case3b}) yields the best trade-off between local accuracy and efficiency for computing reachable sets.
The answer is not straightforward and most likely depends on the system itself.
In future work, we plan to investigate the efficiency of the algorithm on the number of parameters for various examples.

 We have only considered differential inclusions in the form of input-affine systems, and give a brief sketch of how these methods can be applied to other classes of system. We also plan to provide a more detailed exposition of the method in these cases. Moreover, the local error that we obtain is a uniform bound for the error in all components. It should be possible to give slightly better componentwise bounds.

\section{ACKNOWLEDGEMENTS}
This research was partially supported by the European Commission through the project ``Control for Coordination of Distributed Systems'' (C4C) as part of the EU.ICT program (challenge ICT-2007.3.7).
%The authors would like to thank Luca Benvenuti for supplying the example in Section~\ref{num}.

\end{document}